\documentclass{amsart}

\usepackage{amsfonts}
\usepackage{amsmath}
\usepackage{amsthm}
\usepackage{graphicx}
\usepackage{mathdots}
\usepackage{mathrsfs}
\usepackage[hidelinks]{hyperref}
\usepackage{comment}
\usepackage[x11names]{xcolor}
\usepackage{stmaryrd} 
\usepackage{tikz-cd}
\usepackage{enumerate}
\usepackage{cite}
\usepackage{orcidlink}

 \usepackage[a4paper,
            bindingoffset=0.2in,
            left=1.5in,
            right=1.5in,
            top=1.5in,
            bottom=1.5in,
            footskip=.25in]{geometry}
\newcommand{\de}{\mathrm{d}}
\newcommand{\R}{\mathbb{R}}
\newcommand{\N}{\mathbb{N}}

\renewcommand{\P}{\mathscr{P}}

\newcommand{\Riz}{\text{Riesz }}

\newcommand{\spaceV}{\mathscr{V}^k}
\newcommand{\testforms}{\mathscr{D}_0^k}

\newcommand{\averaging}{\mathscr{A}^k}
\renewcommand{\L}{\mathscr{L}}

\newcommand{\revised}[1]{#1}

\DeclareMathOperator*{\Span}{span}

\DeclareMathOperator*{\support}{supp}

\DeclareMathOperator*{\argmin}{arg \ min}

\DeclareMathOperator*{\Card}{Card}
\DeclareMathOperator*{\sgn}{sgn}

\DeclareMathOperator*{\interior}{int}

\DeclareMathOperator*{\vdm}{vdm}

\DeclareMathOperator*{\Tan}{Tan}

\newcommand{\Jac}[1]{\llbracket #1\rrbracket}
\newcommand{\opnorm}[1]{\Vert #1 \Vert_{\mathrm{op}}}

\newcommand{\haus}{\mathcal{H}}

\newtheorem{theorem}{Theorem}
\newtheorem{proposition}{Proposition}
\newtheorem{lemma}{Lemma}
\newtheorem{corollary}{Corollary}
\theoremstyle{definition}
\newtheorem{definition}{Definition}

\theoremstyle{remark}
\newtheorem{remark}{Remark}
\newtheorem{example}{Example}

\newcommand{\myqed}{$ $}
\newtheorem*{theorem*}{Theorem}

\title{The Lebesgue constant for uniform approximation of differential forms}

\author{Ludovico Bruni Bruno}
\address{Dipartimento di Matematica \textquotedblleft Tullio Levi-Civita\textquotedblright, Università di Padova, via Trieste, 63, Padova, 35131, Italia \\
              Istituto Nazionale di Alta Matematica \textquotedblleft Francesco Severi\textquotedblright, Piazzale Aldo Moro, 5, Roma, 00185, Italia}
\email{ludovico.brunibruno@unipd.it}

\author{Federico Piazzon}
\address{Dipartimento di Matematica \textquotedblleft Tullio Levi-Civita\textquotedblright, Università di Padova, via Trieste, 63, Padova, 35131, Italia}
\email{fpiazzon@math.unipd.it}

\begin{document}
\maketitle

\begin{abstract}
    In this work we address the problem of uniform approximation of differential forms starting from weak data defined by integration on rectifiable sets. We study approximation schemes defined by the projection operator $L$ given by either generalized weighted least squares or interpolation. We show that, under a natural measure theoretic condition, the norm of such operator equals the Lebesgue constant of the problem. We finally estimate how the Lebesgue constant varies under the action of smooth mappings from the reference domain to a \textquotedblleft physical\textquotedblright\ one, as is customarily done e.g. in finite elements method. 
\end{abstract}

\section{Introduction}

Differential forms are sections of the $k$-th exterior power of the cotangent bundle of a manifold \cite[\S 14]{Lee}. Hence, they represent vector valued \textquotedblleft functions\textquotedblright\ endowed with an alternating product. This algebraic structure is what makes differential forms tailored for integration: an evocative description is due to Flanders, who presents them as \textquotedblleft things which occur under the integral sign\textquotedblright\ in \cite[p. 1]{Flanders}. This very informal definition stresses how deeply the development of differential forms is intertwined with integration theory and geometric measure theory, and also clarifies why they are often exploited to relate the local and the global behavior of physical phenomena. As a consequence, differential forms are generally exploited to embody the geometry of the domain under consideration into the physics of problems: this perspective is influencing the numerical community as well, from finite elements approximation \cite{ArnoldFalkWintherActa} and mimetic methods \cite{KreeftGerritsma} to, more generally, structure preserving methods for PDEs \cite{Hiemstra}. Especially in this latter context, linear functionals that are represented by an oriented geometrical support in the physical space where integration of the differential form is performed \revised{become relevant}. Such degrees of freedom assume the concrete meaning of circulations along lines (e.g., for the electric field), fluxes across faces (e.g., for the Stokes' flow or the magnetic field), and so on. The extensive and profitable use of differential forms in the aforementioned contexts lead mostly to a study of convergence in Sobolev norms (see, e.g., \cite{GawlikLicht}), which is suggested by regularity theory and functional analysis. On the contrary, a thorough study of continuous differential forms under the light of integration and uniform approximation theory is still lacking. \revised{Under appropriate hypotheses, these two approaches may be related using results of \cite{Stern}.}

In this work we address the study of the uniform approximation of a differential form\revised{ $\omega$ } with continuous coefficients \revised{over a real body $E\subset \R^n$ (i.e., the closure of a bounded domain) } starting from weak data, i.e., the values attained by certain given continuous functionals $\mathcal T:=\{T_i\}_{i=1}^M $ on $\omega$, \revised{with $ M \in \N \cup +\infty$}. The word \textquotedblleft uniform\textquotedblright\ refers to the use of the\revised{ \emph{zero norm} $\|\cdot\|_0$, }introduced in Subsection \ref{subsec:norm}\revised{, that turns the space of continuous differential forms into a Banach space denoted by $\testforms(E)$. }In fact, $\|\cdot\|_0$ proves to be a suitable extension of the uniform norm of continuous functions to the framework of differential forms.\revised{
What makes this norm natural in the present framework is that it involves the action of the differential form on simple $k$-vector fields, tailoring it for working with integral functionals.}

It goes without saying that there is plenty of freedom in the choice of which class of functionals should be considered: in the present work we assume that the functionals in the collection $ \mathcal T $ are currents \cite{deRham} represented by integration.
\revised{Precisely, we focus on the set $\averaging(E)$ of $k$-rectifiable \emph{averaging} currents. Rougly speaking, these functional are measure theoretic generalizations of integral means along oriented manifolds. The precise definition of this set of functionals is stated in Equation \eqref{integralaverage}  of Section \ref{subsec:currents}, while a brief review of integration of differential forms along rectifiable sets is given in Section \ref{sec:rectifiableset}.} 
  \revised{ The choice of considering currents in $\averaging(E)$ as sampling operators } is quite general and rather natural, due to the entanglement of differential forms with integration. Further, it generalizes previous approaches such as Lagrange interpolation by weights \cite{RapettiM2AN} and applies to certain finite volumes schemes \cite{Wendland}.
 
The uniform approximation schemes we study rely on a projection operator $ L : \mathscr D_0^k(E) \to \spaceV(E) $ onto a given $N$-dimensional subspace $\spaceV(E) \subset \mathscr D_0^k(E)  $.  
Relevant instances of finite dimensional subspaces $\spaceV(E)$ our results apply to are forms with polynomial coefficients (N\'ed\'elec second family \cite{Nedelec2}) or trimmed polynomial forms (N\'ed\'elec first family \cite{Nedelec1}, usually also termed \emph{high order Whitney forms} \cite{HiptmairPIER}) customarily used in finite elements exterior calculus \cite{ArnoldFalkWintherActa}, and whose  duality with usual objects of vector calculus is discussed in \cite{Gopalakrishnan}.

\revised{In the setting we propose, $ L $ may be either defined by generalized interpolation (denoted by $L=\Pi$, see   \eqref{eq:interpolationconditions}) of the linear functionals $\mathcal T=\{T_1,\dots,T_N\}$, or by weighted least squares fitting (denoted by $L=P$, see \eqref{discreteleastsquares}) of the (at most) countable family of linear functionals $\mathcal T=\{T_i\}_{i=1}^M $, $M\in \N\cup \{+\infty\}$, with respect to the sequence of positive weights $ \{w_i\}_{i=1}^M $, that we will always assume to satisfy
$$w\in \ell^1(\N),\;\text{ if }M=+\infty.$$ }
Both schemes are formally introduced in Section \ref{sect:problems}, and extend classical approximation theory in several directions: our analysis allows for countable sets of sampling functionals
, merely requiring that\revised{ the supports $S_i$ of the currents $T_i$ }are $k$-rectifiable sets of finite Hausdorff measure $\haus^k$, and has no restrictions on the order of the differential form.

If approximation is carried by projection, the norm of the projection operator plays a prominent role. Indeed, on the one hand $\opnorm{L}$ controls the stability of the fitting procedure via the inequality
$$ \Vert L \omega - L \widetilde{\omega} \Vert_0 \leq \Vert L \Vert_{\mathrm{op}} \|\omega-\widetilde\omega\|_0 ,$$
and on the other hand it is determining for the quality of the obtained approximation via the Lebesgue inequality
$$ \Vert \omega - L \omega \Vert_0 \leq \left(1 + \Vert L \Vert_{\mathrm{op}}\right)  \min_{\eta \in \mathscr V^k (E)} \Vert \omega - \eta \Vert_0 .$$
The above inequalities make self-evident the importance of computing, or at least estimating, the quantity $ \Vert L \Vert_{\mathrm{op}} $. To this end, we introduce the \emph{Lebesgue constant}
\begin{equation} \label{eq:defconstantM}
    \L(\mathcal T,\spaceV(E)) :=\sup_{T\in \averaging(E)} \sum_{i=1}^{M}\left| \sum_{h=1}^NT_i(\eta_h)T(\eta_h)\right|w_i,\;\;\;\revised{M\in \N\cup \{+\infty\}} \,,
\end{equation}
where $ \{\eta_1, \ldots, \eta_N\} $ is any orthonormal basis of $ \spaceV (E) $ with respect to the 
$$(\eta, \vartheta)_{\mathcal T, w}:=\sum_{i=1}^{M}T_i(\eta)T_i(\vartheta)w_i. $$
\revised{We remark that we will carry out our analysis in the setting $M=+\infty$, as the same results easily follow in the finite case. Indeed, if $\mathcal T=\{T_i\}_{i=1}^M$, $w=\{w_i\}_{i=1}^M$, with $M<+\infty$, we can consider the sequence $\widetilde{\mathcal T}$ obtained from $\mathcal T$ by padding it by zero currents and any sequence of weights $\tilde w=\{\tilde w_i\}_{i=1}^M\in \ell^1(\N)$ such that $\tilde w_i=w_i$ for any $i\leq M$. Nevertheless, we will sometimes specialize our results to the relevant setting of interpolation, where in fact $M=N<+\infty.$}

The quantity $ \L(\mathcal T,\spaceV(E)) $ entangles geometrical and analytical aspects of the currents in $ \mathcal{T} $ and the space $ \spaceV (E) $. As first main result of the paper, we extend the well-known equality between the norm of the projection operator $ P $ and the Lebesgue constant from the nodal fitting of functions to the general framework described above.
\begin{theorem}[Characterization of the norm of the projection operator]\label{thm1}
Let $\mathcal T=\{T_1,T_2,\dots\}$  be a countable $\spaceV(E)$-determining subset of $\averaging(E)$ and $w\in \ell^1(\N)$. If $\haus^k(\support T_i \cap \support T_j)=0$ for $i\neq j$, then the discrete least squares projection $P$ satisfies
\begin{equation}\label{Pnormestimate}
\opnorm{P}= \L(\mathcal T,\spaceV (E))\,.
\end{equation}
Inequality  $\opnorm{P}\leq\L(\mathcal T,\spaceV (E))$ still holds even if $\haus^k(\support T_i \cap \support T_j)>0$ for some $i\neq j$.
\end{theorem}

The proof of Theorem \ref{thm1} is carried in Section \ref{sec:3} by an approximation argument in which countable subsets of $ \averaging (E) $ are identified with linear operators from $ \testforms (E) $ to $ \ell^1_w (\N) $. As a by-product, we obtain the continuity of the Lebesgue constant with respect to the point-wise convergence of $ \mathcal{T}^{(\varepsilon)} \to \mathcal T$ \revised{(seen as linear operators from $ \testforms (E) $ to $ \ell^1_w (\N) $) } as $\varepsilon\to 0$, see Proposition \ref{prop:continuityL}, while the uniform convergence of $ \mathcal{T}^{(\varepsilon)} $ implies the strong convergence of the induced projection operators, see Corollary \ref{cor:continuityoperators}.

In the interpolatory case, when in particular the cardinality of $ \mathcal T $ is $ N < +\infty $, the orthonormal basis reduces to\revised{ a scaling of }the Lagrange (cardinal) basis $ \{ \omega_1, \ldots, \omega_N \} $ defined in  \eqref{eq:LagrangeRepresentationWF}\revised{: namely we have $\eta_h=\omega_h/\sqrt{w_h}$, for $h=1,\dots,N,$ }and the Lebesgue constant introduced in   \eqref{eq:defconstantM} assumes the peculiar form
\begin{equation} \label{eq:defconstantL}
    \L(\mathcal T,\spaceV (E)) =\sup_{T\in \averaging(E)} \sum_{i=1}^N \left|  T(\omega_i)\right|.
\end{equation} 
In particular, when the supports of the currents are taken to be simplices, the quantity 
\eqref{eq:defconstantL} 
corresponds to the Lebesgue constant introduced in \cite{AlonsoRapettiLebesgue} in the context of Whitney forms. 
This is an instance of a more general relationship, which is discussed in detail in Section \ref{sect:MtoN}. In fact, in the interpolatory case, Theorem \ref{thm1} specializes as follows:
\begin{corollary} \label{cor:norminterp}
Let $\mathcal T=\{T_1,\dots,T_N\}\subset \mathscr \averaging^k(E)$ be a $\spaceV (E) $-unisolvent set. The interpolation operator $\Pi$ satisfies
\begin{equation}\label{Pinormestimate}
\opnorm{\Pi}\leq \L(\mathcal T,\spaceV (E)).
\end{equation}
Further, equality in   \eqref{Pinormestimate} is attained in the case $\haus^k(\support T_i \cap \support T_j)=0$ for $i\neq j$. 
\end{corollary}

The nodal Lebesgue constant depends only on the positioning of points inside a fixed domain: transforming appropriately the points and the domain does not change the precomputed value. In many applications, such as finite elements approximation and its curvilinear variant (see, e.g., \cite{ArBoBo15}), it is customary to fix a reference element $ \widehat E $, where most computations are performed, and then to transfer the obtained results to the physical domain  $ E $ via an appropriate mapping, say a $\mathscr C^1$-diffeomorphism $\varphi:\widehat E \rightarrow E.$ It is then natural to investigate how 
$ \L $ is affected by this procedure. 

To this end, suppose to be given a space $ \spaceV (\widehat E) \subset\testforms(\widehat{E})$ on the reference element, endowed with a $ \spaceV (\widehat E) $-determining set of currents\revised{ $\widehat {\mathcal T} = \{ \widehat T_i\}$ and a relative positive summable sequence of weights $\widehat w=\{ \widehat w_i\}$, see   }\eqref{eq:determining}. 
The pullback $ \varphi^*: \mathscr D_0^k( E) \to \mathscr D_0^k(\widehat E) $ defines a space $\mathscr V^k(E):=\{ \omega \in \mathscr D_0^k(E) \,:\, \varphi^*\omega\in \mathscr V^k(\widehat E)\}$ on the physical element, 
which inherits from $\spaceV(\widehat E)$ the scalar product 
$ ( \omega, \eta)_{\mathcal T,w} : = (\varphi^*\omega,\varphi^*\eta)_{\widehat{\mathcal T},\widehat w}$, and hence a least squares projector $ P $. The quantities\revised{ ${\mathcal T}=\{T_i\}$ and $w=\{ w_i\}$ }are suitable renormalizations of the corresponding quantities on $ \widehat E $
, and are discussed in Subsection \ref{sect:pushforward}. 
Then the following relationship holds:
\begin{theorem}[Dependence of $ \L $ on $ E $] \label{thm2}
Let $ \varphi:\widehat E\rightarrow E $ be a $\mathscr C^1$-diffeomorphism. Then one has
	\begin{equation} \label{eq:estimationchangingsimplexI}
		\L ({\mathcal T}, \mathscr V^k (E) ) \leq \left\|\prod_{j=1}^k \sigma^{(\varphi)}_j\right\|_{\mathscr C^0(\widehat E)} \left\| \frac 1{\prod_{j=n-k+1}^n\sigma^{(\varphi)}_j}\right\|_{\mathscr C^0(\widehat E)} \L (\widehat{\mathcal T}, \mathscr V^k (\widehat{E}) )\,,
	\end{equation}
where $\left\{\sigma^{(\varphi)}_1,\dots,\sigma^{(\varphi)}_n\right\}$ are the singular value functions of the differential of $\varphi$ arranged in a non-increasing order, and $\|\cdot\|_{\mathscr C^0(\widehat E)}$ denotes the uniform norm of continuous functions on $\widehat E$.
\end{theorem}

Theorem \ref{thm2} applies to the interpolatory case and is a uniform nonlinear counterpart of the error estimate frequently invoked in local error analysis of finite elements \cite[\S 11.2]{ErnI}, where affine mappings are considered. Further, it applies not only to functions but to differential forms of any order $k$. In this spirit, a closer look to   \eqref{eq:estimationchangingsimplexI} unveils that the ratio of the Lebesgue constants is bounded by the conditioning of the linear map $ D \varphi $ raised to the power $k$, which essentially captures the complexity of the underlying geometry of the problem. A more detailed discussion on the consequences of Theorem \ref{thm2} is developed in Section \ref{sect:changing}.

\section{Background} \label{sec:Background}

\revised{
The aim of this section is to review the basic mathematical background, mostly regarding integration of differential forms on rectifiable sets, currents, and the norms on differential forms, we shall extensively use in the subsequent sections.

\subsection{Vectors, covectors and norms} We use the symbol $\Lambda_k$ for the $k$-th exterior power of $\R^n$. By construction, $ \Lambda_k $ is an $\binom{n}{k}$-dimensional vector space, which is endowed by the Euclidean scalar product $(\cdot,\cdot)_{\Lambda_k}$, see, e.g., \cite[p. 16]{Bhatia}. The corresponding norm $ | \cdot |_{\Lambda_k} $ is the Euclidean norm for $k$-vectors. 
Note that, in particular, $ |\cdot|_{\Lambda_0} $ is the absolute value $ \vert \cdot \vert $ and $ |\cdot|_{\Lambda_1} $ is the Euclidean distance $ \Vert \cdot \Vert $ between elements of $ \R^n $. 
Similarly, $ \Lambda^k $ denotes the $k$-th exterior power of the space of linear forms on $\R^n$. This space is endowed by the Euclidean scalar product $(\cdot,\cdot)_{\Lambda^k}$, and we denote the corresponding (Euclidean) norm by $ | \cdot |_{\Lambda^k} $.

Besides the Euclidean norm, on $\Lambda^k$ we will also consider the \emph{comass norm} $ | \cdot |^* $. Recalling that a $k$-vector $ \tau $ is \emph{simple} if it can be written in the form $ \tau = v_1 \wedge \ldots \wedge v_k $, with $ v_i \in \R^n $ for each $ i$, such a norm is defined as
$$|\omega|^*:=\sup\{|\langle\omega;\tau\rangle|,\;\tau\in \Lambda_k \text{ simple},|\tau|_{\Lambda_k}\leq 1\}. $$
Here and throughout the paper we denote by $\langle\cdot\,;\cdot\rangle$ the dual pairing of $\Lambda^k$ with $\Lambda_k.$

By definition, the comass norm $|\cdot|^*$ of a $k$-covector $\omega$ cannot exceed its Euclidean norm $|\omega|_{\Lambda^k}$, equality being attained only when $\omega$ is a simple $k$-covector. The converse inequality also holds up to the factor $\sqrt{\binom{n}{k}}$ by H\"older Inequality. This yields the comparability of the two norms, i.e., for each $\omega\in \Lambda^k$
\begin{equation}\label{eq:euclideanvscomass}
|\omega|^*\leq |\omega|_{\Lambda^k}\leq \sqrt{\binom{n}{k}}|\omega|^*\,.
\end{equation}
Improved comparability constant can be given \cite{comass}. It is also worth mentioning here that there are cases for which all $k$-covectors are simple  (e.g., when $k\in\{0,1,n-1,n\}$ and arbitrary $n\in \N$, or when $n\leq 3$ and arbitrary $k\leq n$) and thus the two norms coincide.



\subsection{Rectifiable sets and integration}\label{sec:rectifiableset}
We recall that a set $S\subset \R^n$ is said to be $k$-\emph{rectifiable} if it can be covered by the union of a set having zero $k$-dimensional Hausdorff measure (hereafter denoted by $\haus^k$) and a countable union of Lipschitz images of $\R^k$ \cite[ Definition 5.4.1]{Krantz}. When $S$ is $\haus^k$-measurable and $k$-rectifiable, then it is possible \cite[Proposition 5.4.3]{Krantz} to write $S=\cup_{i=0}^{+\infty} S_i$, where $S_i\cap S_j=\emptyset$ whenever $i\neq j$, $\haus^k(S_0)=0$, and, for all $i>0$, $S_i$ is a measurable subset of a $\mathscr C^1$ embedded $k$-submanifold of $\R^n$. 

The standard definition of tangent space is not well-suited for sets that do not admit a differential structure. In fact, in geometric measure theory it is customarily used a more relaxed notion of tangent space, given by functional and measure theoretic means. Namely, if $S$ is a $\haus^k$-measurable set of locally finite $\haus^k$ measure, we term a $k$-dimensional linear subspace $V$ of $\R^n$ the \emph{approximate tangent space of}$S$\emph{ at }$x$ if, for any continuous and compactly supported function $f\in \mathscr C^0_c(\R^n)$, one has
$$\lim_{t\to 0^+}\int_{\{y\in \R^n: \ ty+x\in S\}}f(y) d\haus^k(y)=\int_Vf(y)d\haus^k(y)\,.$$
If such a $V$ exists, then it is uniquely determined, and we use the notation $\Tan(S,x)$ for it. This notion of tangent space generalizes the definition of the tangent space of a $k$-submanifold of class $\mathscr C^1$, as the two notions coincides if $S$ bears such a regularity. In contrast, it is rather clear that $\Tan(S,x)$ might not exist for some $x\in S$ for less regular sets $S\subset \R^n$. Remarkably, if $S$ is a $k$-rectifiable set of locally finite measure, then $\Tan(S,x)$ does exist for $\haus^k$-almost every $x\in S$, see \cite[Theorem 5.4.6]{Krantz}.
\begin{definition}[Orientation of a $k$-rectifiable set]\label{def:orientation}
An \emph{orientation} for the $\haus^k$-measurable $k$-rectifiable set $S$ is a \emph{simple} $k$-vector field $S\ni x\mapsto \tau(x)=\tau_1(x)\wedge\dots\wedge \tau_k(x)\in \Lambda_k$ such that, for $\haus^k$-almost every $x\in S$, $|\tau|_{\Lambda_k}=1$ and $\Tan(S,x)=\Span\{\tau_1(x),\dots,\tau_k(x)\}$. The $\haus^k$-measurable $k$-rectifiable set $S$ is said to be \emph{oriented} if it is endowed by an orientation. 
\end{definition}
\begin{remark}
It is worth pointing out here that the concept of orientation given in Definition \ref{def:orientation} is applicable only to $k$-rectifiable sets that have locally finite measure, since the definition of the approximate tangent space relies on such an assumption. For this reason, when we declare that a set $S$ is oriented, we also  implicitly assume that it has locally finite $\haus^k$-measure.  
\end{remark}

We may then define the \emph{integral} on the $k$-rectifiable set $S$ with respect to the orientation $\tau$ {of a $k$-form} $\omega:\R^n\rightarrow \Lambda^k$ with continuous\footnote{The notion of continuity here does not depend on the choice of a basis of $\Lambda^k$, neither on the particular choice of a norm on $\Lambda^k$, being the latter a finite dimensional vector space.} and compactly supported coefficients  as
\begin{equation}\label{eq:defineintegration}
\int_S \langle\omega(x);\tau(x)\rangle \de \haus^k(x)\, .
\end{equation}
\begin{remark}\label{rem:integrationofcontinuous}
When $S$ has in fact \emph{finite} $\haus^k$-measure, the hypothesis on the compactness of the support of $\omega$ can be dropped, and thus \eqref{eq:defineintegration} extends to continuous forms.
\end{remark} 
\begin{remark}[Orientability of $k$-rectifiable set of locally finite measure]\label{rem:orientability}
The concept of orientability usually applied to differentiable manifolds is quite stiff: indeed, when $S$ is regarded as a $k$-submanifod of $\R^n$, either it is orientable (in exactly two possible ways) or not.  
In constrast, in the framework of $k$-rectifiable sets  one can always define an orientation (in the sense of Definition \ref{def:orientation}), provided the set under consideration has locally finite $\haus^k$-measure, see, e.g., \cite[Proposition 1.3.4]{Ma13}. 
\end{remark}
The following two examples give an account of Remark \ref{rem:orientability}. 
\begin{example}
Consider the $1$-submanifold $S=(0,1)$ of $\R$. Let $\phi\in\mathscr C^0_c(\R,[0,1])$,with $\phi(x)=1$, for any $x\in S$, and let $\omega=\phi(x) dx.$
Then $\int_S\omega=\pm 1$, depending on the choice of the orientation of the manifold $S$.

In contrast, if we regard $S$ as $k$-rectifiable set oriented by the $\haus^k$-measurable function $\tau:S\rightarrow \{-1,1\}$, then there exist two $\haus^k$-measurable sets $S_+,S_-\subset S$ such that $\tau(x)=1$ for any $x\in S_+$, $\tau(x)=-1$ for any $x\in S_-$, and $\haus^k(S\setminus(S_+\cup S_-))=0$. Thus we have
 $$\int_S\langle\omega(x);\tau(x)\rangle d\haus^k(x)=\haus^k(S_+)-\haus^k(S_-)\,,$$
 a quantity that can be any real number in $[-1,1]$ depending on $\tau$.
\end{example} 
\begin{example}
Let $\{x_i\}_{i=1}^{+\infty}$ be any enumeration of $\mathbb Q\times\mathbb Q$, $\{l_i\}_{i=1}^{+\infty}$ any positive summable sequence, $\{\vartheta_i\}_{i=1}^{+\infty}$ any sequence in $[0,2\pi]$. Let 
$$S:=\bigcup_{i=1}^{+\infty}\Big\{x\in \R^2: x=x_i+l(\cos\vartheta_i,\sin\vartheta_i)^\top, l\in[0,l_i]\Big\}\,.$$
By construction, the set $S$ is a $1$-rectifiable dense subset of $\R^2$. Also, $S$ has finite (and in particular locally finite) $\haus^1$-measure, since $\haus^1(S)=\sum_{i=1}^{+\infty}l_i<+\infty$. Hence, for $\haus^1$-almost every $x\in S$ the vector space $\Tan(S,x)$ is well defined (see, e.g., \cite[Theorem 5.4.6]{Krantz}), and it necessarily coincides with $\{t (\cos\vartheta_i,\sin\vartheta_i)^\top,\,t\in \R\}$ for some $i(x)\in \N$. Finally, we can consider on $S$ the orientation
$$\tau(x):=\begin{cases}
(\cos\vartheta_{i(x)},\sin\vartheta_{i(x)})^\top& \text{ if }\Tan(S,x)\text{ is well defined}, \\
0&\text{ otherwise}.
\end{cases}$$
\end{example}

\subsection{The zero norm, a \textquotedblleft uniform\textquotedblright\ norm for differential forms}\label{subsec:norm}

From now on, we assume that $ E \subset \R^n $ is the closure of a bounded domain. 
\begin{definition}[Test forms]
We call $\testforms(E):=\left( \mathscr C^0(E,\Lambda^k),\|\cdot\|_0\right)$ the space of \emph{test forms}, which is the set of continuous differential forms of order $k$ endowed with the \emph{zero norm} $\|\cdot\|_0$, defined as
\begin{equation}\label{0normdef}
\|\omega\|_0:=\sup\frac{1}{\haus^k(S)}\left|\int_S\langle\omega(x);\tau(x)\rangle d\haus^k(x)\right|\,,
\end{equation}
where the supremum is sought among all $k$-rectifiable $S\subset E$ with orientation $\tau$ such that $ 0<\haus^k(S) < + \infty$.
\end{definition}
The quantity \eqref{0normdef} is well defined in view of Remark \ref{rem:integrationofcontinuous}, and such a norm makes $\testforms(E)$ a Banach space. Moreover, we point out that the absolute value appearing in \eqref{0normdef} can be removed, as  $-\tau$ is an orientation of $S$ whenever $\tau$ is. We decide to keep such absolute value to stress the positivity of the right hand side of \eqref{0normdef}.
\begin{remark}
The zero appearing in the notation $\testforms(E)$ stands for the continuity of the elements of the space, not for the compactness of their support.
Also, our definition of test forms differs from that commonly used in calculus of variation, as in the present work $E$ is always the closure of a bounded domain, and thus compact. Notice that, if $E$ were an open set, the space of test forms commonly considered in the literature would be defined as the locally convex topological vector space arising as the strict inductive limit of certain Frech\'et spaces (see, e.g., \cite[\S 2.6]{Sc99} or \cite[\S 2]{Bo17}); this procedure does not lead in general to the construction of a Banach space.   
\end{remark}
Our interest in the zero norm has a twofold motivation. Indeed, on the one hand, the zero norm 
is well-suited to study integral functionals acting on differential forms \cite{Harrison}, as those we will use as sampling operators later on. On the other hand, this choice allows to relate our study with the literature \cite{AlonsoRapettiLebesgue}, bridging approximation theory with tools of geometric measure theory.

Despite the integration appearing in \eqref{0normdef}, the norm $\|\cdot\|_0$ should be regarded as an appropriate generalization of the uniform norm to the space of differential forms, as the following example shows.

\begin{example} \label{ex:norms}
	When $ k = 0 $, $ \haus^k $ is the counting measure and any oriented $k$-rectifiable set $ S $ of finite $\haus^k$-measure is a finite collection of points $ \{ \xi_i \}_{i=1}^M $ in $ E $ with orientation $ \tau_i = \pm 1 $ for $ i = 1, \ldots, M $. Further, since $ \Lambda^0 = \R $, $ \omega \in \mathscr{D}_0^0 (E) $ is a real-valued continuous function. One has
	$$ \|\omega\|_0 =\sup\left\{\frac 1{M} \left| \sum_{i=1}^M \tau_i \omega (\xi_i) \right|, \ M \in \N, \ \xi_i \in E, \ \tau_i = \pm 1 \right\}\, = \max_{\xi \in E} \left| \omega (\xi) \right| .$$
	When $ k = n $, $ \haus^k $ coincides with the $n$-dimensional Lebesgue measure, and one represents $ \omega \in \mathscr D_0^n(E)$ as $ \omega = f \de x_1 \wedge \ldots \wedge \de x_n $ for some $ f \in \mathscr{C}^0 (E) $. Then the estimate
\begin{align*}
&\left|\int_S \langle\omega(x);\tau(x)\rangle d\haus^n(x)\right|=\left|\int_S f(x) \det[\langle dx_i;\tau_j(x)\rangle]_{i,j=1,\dots,n} d\haus^n(x)\right|\\
\leq& \int_S |f(x)| |\det[\langle dx_i;\tau_j(x)\rangle]|_{i,j=1,\dots,n} d\haus^n(x)
\leq \max_{x\in E}|f(x)|\haus^n(S)\,,
\end{align*}
holds for any $k$-rectifiable set $S$ with orientation $\tau=\tau_1\wedge\dots\wedge\tau_n$. Hence 
$$\|\omega\|_0\leq \max_{x\in E}|f(x)|=\max_{x\in E}|\omega(x)|^*\,,$$
since all $n$-vectors are simple in $\R^n$.

Conversely, if $\bar x\in E$ is such that $|f(\bar x)|=\max_{x\in E}|f(x)|$, then for any $\varepsilon>0$ small enough, we can consider $S_\varepsilon:=B(\bar x,\varepsilon)\cap E$ and $\tau(x)\equiv\sgn f(\bar x) e_1 \wedge \ldots \wedge e_n $ and notice that
$$\lim_{\varepsilon\to 0^+}\frac 1 {\haus^k(S_\varepsilon)}\int_{S_\varepsilon}\langle\omega(x);\tau(x)\rangle d\haus^n(x)=|f(\bar x)|\,.$$
%
%
%
\end{example}

For the intermediate cases $ 0 < k < n $ we may still regard $ \Vert \cdot \Vert_0 $ as an extension to differential forms of the uniform norm. Precisely, under the regularity assumptions that we made on $E$, one has
\begin{equation}
\label{normequivalence2}
\|\omega\|_0=  \max_{x\in E}|\omega(x)|^*.
\end{equation}
Note that \eqref{normequivalence2}, together with the comparability of $|\cdot|^*$ with $|\cdot|_{\Lambda^k}$ stated in \eqref{eq:euclideanvscomass}, implies the topological equivalence of the zero norm with the uniform norm of the point-wise Euclidean norm:
$$\frac{1}{\sqrt{\binom{n}{k}}}\max_{x\in E}|\omega(x)|_{\Lambda^k}\leq \|\omega\|_0\leq \max_{x\in E}|\omega(x)|_{\Lambda^k},\;\;\forall \omega\in \testforms(E)\,.$$
\begin{remark} \label{rmk:simplicialapproximation}
Under the assumptions of our framework, working by simplicial approximation, it is possible to show that
\begin{equation}
\label{normequivalence}
\|\omega\|_0=\sup_{S\in \mathcal S^k(E)}\frac 1{\haus^k(S)} \left| \int_S \omega \right|,
\end{equation}
where $ \mathcal S^k(E):=\left\{\text{all }k\text{-dimensional oriented simplices lying in }E \right\} $. Further, it is also possible to consider simplicial chains supported on $ E $, as done in \cite{Harrison} and exploited in \cite{AlonsoRapettiLebesgue} to define a Lebesgue constant.
\end{remark}
\begin{remark}
We point out that the equivalent characterizations stated in  \eqref{normequivalence2} and  \eqref{normequivalence} above strongly depend on the assumption we made on the compact set $E$, namely $E$ is the closure of a bounded domain. Indeed, these characterization may fail for more general compact sets. For, simply consider $E:=(\{0\}\times [-1,1]) \cup ([-1,1]\times \{0\})$ and $\omega:=y(1-x^2) dx$
and note that $\|\omega\|_0=0$ (since $\omega$ has vanishing integral on any subset of $E$, while $\max_{x\in E}|\omega(x)|^*\geq |\langle (1,0),\omega(0,1)\rangle|=1$ . 
\end{remark}

}

\subsection{Sampling differential forms by currents of order zero}\label{subsec:currents}
In the next sections we will set up and analyze certain approximation schemes for differential forms in $\testforms(E)$ based on generalized interpolation or discrete least squares fitting. In order to ensure well-posedness of such procedures, we clearly need to consider sampling strategies that arise from linear continuous functionals acting on $\testforms(E)$. \emph{Currents of order zero} (and dimension $k$) are the elements of the dual space $ (\testforms(E))' $ of $ (\testforms(E), \Vert \cdot \Vert_0) $.
Due to the duality relationship, $ (\testforms(E))' $ is naturally endowed with the operator norm
\begin{equation*}\label{eq:mass}
M(T):=\sup\{|T(\omega)|, \; \omega\in \mathscr D_0^k(E),\;\|\omega\|_0\leq 1\},
\end{equation*}
which is referred to as the \emph{mass} of the current $T$. 

As pointed out by Klimek, \textquotedblleft Currents of order zero may be regarded as differential forms with measure coefficients\textquotedblright\ \cite{Kl91}. Indeed, by Riesz Representation Theorem, for any element $T$ of $(\testforms(E))'$ there exist a Borel regular\revised{ finite }measure $\mu$ and a $k$-vector field $\tau \in L^\infty_\mu(E)$, with $|\tau|_{\Lambda_k}=1$ $\mu$-a.e., such that 
\begin{equation}
    T(\omega)=\int \langle \omega(x); \tau(x) \rangle d\mu(x),\;\;\forall \omega\in \testforms(E)\,.
\end{equation}
The integration over a smooth embedded oriented $k$-submanifold of $\interior E$ with finite $k$-dimensional measure defines a current of order zero and dimension $k$. Nevertheless, smoothness is not really playing a relevant role in such a construction, and one may work with less regular objects, namely oriented $k$-rectifiable sets of finite $\haus^k$ measure.

We are in fact interested in two relevant subclasses of $ (\testforms(E))' $ defined via \eqref{eq:defineintegration}. The first one is exactly that of \emph{currents of integration} $\mathscr I^k(E)$, i.e., any current $[S,\tau] \in (\testforms(E))' $ defined by integration over an $\haus^k$-measurable $k$-rectifiable oriented subset $S \subset E $:
\begin{equation}\label{integralcurrent}
[S,\tau](\omega) :=\revised{\int_S\langle\omega(x);\tau(x)\rangle d\haus^k(x)}\,.
\end{equation}

The second subclass of $ (\testforms(E))' $ we consider here is the one of \emph{currents of integral averaging} $\averaging(E) $, or \emph{averaging currents} for short.
\revised{
\begin{definition}[Averaging currents]
The set of $\averaging(E) $ \emph{averaging currents} is the subset of all $T\in (\testforms(E))'$ such that there exists a $k$-rectifiable set $S$ of positive finite measure and an orientation $\tau $ of $S$ such that, for any $\omega\in \testforms(E)$, one has
\begin{equation}\label{integralaverage}
T(\omega):=\frac{[S,\tau](\omega)}{\haus^k(S)}=\frac 1 {\haus^k(S)}\int_S\langle\omega(x);\tau(x)\rangle d\haus^k(x) \, .
\end{equation} 
\end{definition}}
Note that in particular one has $M([S,\tau])=\haus^k(S)$, and, if $T\in \averaging(E)$, then $M(T)=1$ by construction. 
Due to this mass bound, any countable set $ \mathcal{T} \subset \averaging (E) $ (possibly right-padded by zero currents if it is finite) can be identified with a linear bounded operator
$$ \mathcal{T}: \testforms (E) \rightarrow \ell^1_w (\N),\;\;\;\mathcal T (\omega)=\{T_i(\omega)\}_{i=1}^{+\infty}\, .$$
In the above equation, $ \ell^1_w (\N) $ is the space real sequences $ y $ such that 
\revised{$$ \Vert y \Vert_{\ell^1_w} := \sum_{i=1}^{+\infty} | y_i | w_i < + \infty\,.$$}
In Section \ref{sec:3}, we will consider families of subsets $ \mathcal{T}^{(\varepsilon)} \subset \averaging (E) $ and define their convergence towards $ \mathcal{T} $ relying on their identification with linear bounded operators mapping $ \testforms(E) $ to $ \ell^1_w (\N) $. In particular, we say that $\mathcal T^{(\varepsilon)}\subset \averaging(E)$ is \emph{point-wise converging} to $\mathcal T\subset \averaging(E)$ if 
\begin{equation}\label{eq:pointwiseconvergence}
\lim_{\varepsilon\to 0}\sum_{i=1}^{+\infty}|T_i(\omega)-T_i^{(\varepsilon)}(\omega)|w_i=0,\;\;\forall \omega\in \testforms(E)\,,
\end{equation}
while we say that $\mathcal T^{(\varepsilon)}\subset \averaging(E)$ is \emph{uniformly converging} to $\mathcal T\subset \averaging(E)$ if 
\begin{equation}\label{eq:uniformconvergence}
\lim_{\varepsilon\to 0} \sup_{\Vert \omega \Vert_0 = 1 } \left( \sum_{i=1}^{+\infty}|T_i(\omega)-T_i^{(\varepsilon)}(\omega)|w_i \right) = 0\,.
\end{equation}

\subsection{The interpolation operator $ \Pi $ and the fitting operator $ P $} \label{sect:problems}
Let $ \spaceV (E) $ be a $N$-dimensional linear subspace of $ \testforms(E) $.  A\revised{ finite or countable }set $\mathcal T:=\{T_i\}\subset (\testforms(E))' $ is termed \emph{determining} for $\spaceV (E) $ whenever \revised{
\begin{equation} \label{eq:determining}
T_i(\omega)=0, \quad \forall i=1,\dots,\Card \mathcal T \qquad \text{implies that} \qquad \omega=0 .
\end{equation}
Necessarily we have $ \Card \mathcal T\geq N $. For, recall that, if $M:=\Card \mathcal T<+\infty$, \eqref{eq:determining} is equivalent to the nullity of the right kernel of the $M$ by $N$ generalized Vandermonde matrix $\vdm(\mathcal V,\mathcal T):=T_i(v_j)$ relative to a basis $v_1,\dots,v_N$ of $\spaceV(E)$.}

\revised{\begin{remark}
Unisolvence depends both on the space $ \spaceV(E) $ and the set $ \mathcal T$; related results and strategies may sensibly vary from case to case. In the case of polynomial differential forms, this can be appreciated by comparing moments-based techniques (used, e.g., in \cite{ArnoldFalkWintherActa}) with weights-based techniques (developed, e.g., in \cite{BruniThesis} for the first family of polynomial differential forms and in \cite{BruniZampa} for the second one).
\end{remark}}

 Now, if \eqref{eq:determining} holds and $ M = N $, the set $\mathcal T$ is said to be \emph{unisolvent} for $\spaceV (E)$,\revised{ and $\vdm(\mathcal V,\mathcal T)$ is invertible for any basis $\mathcal V$ of $ \spaceV (E) $. }In such a case,  we may generalize the idea of \cite{ABRCalcolo} and define an interpolation operator $ \Pi: \testforms(E) \to \spaceV(E) $ by asking that
\begin{equation} \label{eq:interpolationconditions}
	T_i (\omega) = T_i (\Pi \omega), \quad i = 1, \ldots, M.
\end{equation}

\revised{Given a countable $ \spaceV (E) $-determining set of currents $\mathcal T:=\{T_i\}$ and a summable sequence of positive weights $w=\{ w_i\} $ we introduce the scalar product $(\cdot \, , \cdot)_{\mathcal T,w}$ on $\spaceV(E)$ by setting 
\begin{equation}\label{scalarproduct}
(\omega,\eta)_{\mathcal T,w}:=\sum_{i=1}^{+\infty}w_iT_i(\omega)T_i(\eta)\,.
\end{equation}}
Note that $(\cdot \, , \cdot)_{\mathcal T,w}$ defines a non-negative symmetric continuous bilinear form on $\testforms(E)$,\revised{ which is in particular a scalar product on $ \spaceV (E) $ (since $\mathcal T$ is determining for such a space) and thus induces a norm on $\spaceV(E)$ that we denote by $\|\cdot\|_{\mathcal T,w}.$ }
Hence, the weighted discrete least squares projector $P:\testforms(E)\rightarrow\spaceV (E) $, 
\begin{equation}\label{discreteleastsquares}
P\omega:= \argmin_{\eta \in \spaceV (E)}\left\{\|\omega-\eta\|_{\mathcal T,w}^2 := (\omega - \eta, \omega - \eta)_{\mathcal T,w} \right\},
\end{equation} 
is well-defined. 

\revised{
\begin{remark}[The finite case $\Card \mathcal T<+\infty$]
Consider a finite $\spaceV(E)$-determining set $\widetilde{\mathcal T}=\{\widetilde T_1,\dots,\widetilde T_M\}\subset \averaging(E)$ and a finite sequence of weights $\widetilde w=\{\widetilde w_1,\dots, \widetilde w_M\}.$ We can embed such a case in our framework simply padding $\widetilde T$ with zero currents, i.e. defining $\mathcal T=\{T_i\}_{i=1}^{+\infty},$ $T_i=\widetilde T_i$ if $i\in\{1,\dots,M\}$, $T_i=0$ for $i>M$, and extending $\widetilde w$ with any positive summable sequence. Hence, the case $ M < + \infty $ can be very often treated as a particular case of the infinite one.
\end{remark}


Using the scalar product defined in \eqref{scalarproduct}, we can compute an orthonormal basis $\{\eta_1,\dots,\eta_N\}$ of $\spaceV (E) $ using,\revised{ e.g., spectral decomposition of the associated Gram matrix with respect to any given basis. } 
Then we can write
\begin{equation} \label{eq:samplingProjector}
P\omega=\sum_{h=1}^N(\omega,\eta_h)_{\mathcal T,w}\eta_h=\sum_{h=1}^N\sum_{i=1}^{+\infty}w_iT_i(\omega)T_i(\eta_h)\eta_h\,,\quad (\eta_i,\eta_j)_{\mathcal T,w}=\delta_{i,j}.
\end{equation}

\revised{\begin{remark}[The case $\Card \mathcal T=N$]
When $ M:=\Card \mathcal T = N $, the set $ \mathcal{T} $ is unisolvent and the mismatch in \eqref{discreteleastsquares} vanishes. Further, when all the weights $ w_i \equiv 1 $, the orthonormal basis $\{\eta_1,\dots,\eta_N\}$ of $\spaceV(E)$ with respect to $(\cdot\,,\cdot)_{\mathcal T,w}$ is uniquely determined up to permutations of its elements. It is precisely the Lagrange basis $\{ \omega_1, \ldots, \omega_N \} $ relative to $\{T_1,\dots,T_N\}$, i.e. it satisfies $ T_i (\omega_j) = \delta_{i,j} $. In such a case $\Pi \omega$ can be conveniently represented in the Lagrange form 
\begin{equation} \label{eq:LagrangeRepresentationWF}
	\Pi \omega = \sum_{i=1}^N T_i (\omega) \omega_i\,,\quad T_i(\omega_j)=\delta_{i,j}\,.
\end{equation}
Notice also that, if we change the weights $w_1,\dots,w_N$, we obtain a different scaling of the orthonormal basis (namely $\eta_i=\omega_i/\sqrt{w_i}$), but the operator $\Pi$ remains the same. 
\end{remark}}
}


\begin{example}[Orthonormality of Whitney forms]
Whitney forms of lowest degree are an established tool in computational electromagnetism \cite[\S 5.2.2]{Bossavit}, due to their geometrical flavor \cite{DodziukCharacterization}. Their formal definition is as follows.

Let $ E \subset \R^n $ be an $n$-simplex spanned by vertices $ \{ v_0, \ldots, v_n \} $. Let $ \{ \lambda_0, \ldots, \lambda_n \} $ be the corresponding barycentric coordinates. Let $\alpha$ be a increasing multiindex of length $k$ and $ E_\alpha := \{ v_{\alpha(0)}, \ldots, v_{\alpha(k)} \} $, with orientation $ \tau_\alpha $ as in \cite[p. 152]{Whitney}. The Whitney $k$-form associated with a $k$-face $ E_\alpha $ of $ E $ is
\begin{equation*} \label{eq:WFdef}
    \omega_\alpha := \sum_{i=0}^k (-1)^i \lambda_{\alpha(i)} \de \lambda_{\alpha(0)} \wedge \ldots \wedge {\de \lambda}_{\alpha(i-1)} \wedge  {\de \lambda}_{\alpha(i+1)} \wedge \ldots \wedge \de \lambda_{\alpha(k)} .
\end{equation*}
One may define the space of Whitney forms
$$ \P_1^- \Lambda^k (E) = \Span \left\{ \omega_\alpha, \, E_\alpha \text{ is a $k$-face of $E$}\right\}$$
and, letting $ T_\beta $ denote the averaging current \eqref{integralaverage} with respect to $ E_\beta $,
$$ [E_\beta, \tau_\beta] (\omega_\alpha) = \haus^k (E_\beta) T_{E_\beta} (\omega_\alpha) = \int_{E_\beta} \omega_\alpha = \haus^k (E_\beta) \delta_{\alpha,\beta} .$$
The basis $ \left\{ \omega_\alpha, \, E_\alpha \text{ is a $k$-face of $E$}\right\} $ of Whitney $k$-forms is orthonormal with respect to the scalar product \eqref{scalarproduct} induced by the $k$-faces $ E_\alpha $ of $ E $ with weights $ w_\theta = 1 $ for each $ \theta $. Indeed:
$$ (\omega_\alpha, \omega_\beta)_{\mathcal{T},w} = \sum_{\theta = 1}^M w_\theta T_\theta (\omega_\alpha) T_\theta (\omega_\beta) = \delta_{\alpha, \beta} .$$
\end{example}

\subsection{Riesz representers and reproducing kernels} In what follows we shall make frequent use of reproducing kernels. Such a well-established theory \cite{ReproducingKernel} gives in turn a more manageable notation. Since $(\spaceV (E),(\cdot,\cdot)_{\mathcal T,w})$ 
is a $N$-dimensional Hilbert space, any linear continuous functional $T$ over $\spaceV(E)$ (and in particular any any zero order current $T$, due to continuous embedding of $\spaceV(E) $ in the Banach space $\mathscr D^k_0$)  admits a \emph{\Riz representer} $K_T$, which is readily computed to be
$$ K_T := \sum_{h=1} T(\eta_h) \eta_h ,$$
where $ \{ \eta_1,\dots,\eta_N \} $ is any orthonormal basis of $\spaceV (E)$. Indeed, for any $\omega\in \spaceV (E) $, we can write
$$T(\omega)=T\left(\sum_{h=1}^N(\omega,\eta_h)_{\mathcal T,w}\eta_h\right)= \sum_{h=1}^N(\omega,\eta_h)_{\mathcal T,w}T(\eta_h)=\left(\omega,\sum_{h=1}^NT(\eta_h)\eta_h\right)_{\mathcal T,w}= (\omega,K_T)_{\mathcal T,w}.$$
In this language, we can rephrase   \eqref{eq:defconstantM} by means of \Riz representers, finding
\begin{equation}\label{eq:rephrasingL}
\L(\mathcal T, \spaceV(E))=\sup_{T\in \averaging(E)}\sum_{i=1}^{+\infty}|T_i(K_T)|w_i=\sup_{T\in \averaging(E)}\sum_{i=1}^{+\infty}|T(K_{T_i})|w_i\,.
\end{equation}

The idea of \Riz representers can be extended also to a linear continuous automorphism $A: \spaceV(E)\rightarrow \spaceV(E)$. To this end, we notice that, for any fixed $ \omega \in \spaceV (E) $, the linear map $ (\cdot,\omega)_{\mathcal T,w} $ can be extended to a linear map $ \spaceV (E) \otimes \spaceV (E) \rightarrow \spaceV (E) $ by setting $ (\eta \otimes \vartheta, \omega)_{\mathcal T, w} := \eta \cdot (\vartheta, \omega)_{\mathcal T, w} $. We can then introduce
$K_A(x,y):=\sum_{h=1}^N A(\eta_h)(x)\otimes \eta_h(y)$ so that
\begin{equation*}\label{eq:representerL}
(K_A(x,y),\omega(y))_{\mathcal T,w}=(A\omega)(x),\;\;\forall \omega\in \spaceV(E),\;\forall x\in E\,.
\end{equation*}
In particular, the identity $ \mathbb{I} $ of $\spaceV(E)$ is represented by the \emph{Reproducing kernel} $ K := K_{\mathbb I} = \sum_{h=1}^N\eta_h(x)\otimes \eta_h(y)$, which is the unique element of $ \spaceV(E) \otimes \spaceV(E) $ such that $ \omega(x)=(K(x,y),\omega(y))_{\mathcal T,w} $ for each $ \omega\in \spaceV $. The reproducing kernel $K$ also allows for an useful equivalent definition of $P$. Indeed the orthogonality of the projection $P$ directly implies
\begin{equation}\label{eq:Pbyscalarproduct}
(P\omega)(x)=(K(x,y),\omega(y))_{\mathcal T,w},\;\;\forall \omega\in \testforms(E),\;\forall x\in E\,.
\end{equation}

\subsection{Equivalent definitions of the Lebesgue constant} \label{sect:MtoN}

In the introduction, we both termed Lebesgue constant two apparently different quantities (compare \eqref{eq:defconstantM} with \eqref{eq:defconstantL}). This is motivated by the following fact. Suppose that $ N = M < +\infty $, and consider the Lagrange basis $ \{ \omega_1, \ldots, \omega_N \} $. Computing the scalar product \eqref{scalarproduct} of two elements of such a set, one finds
$$ (\omega_j, \omega_h)_{\mathcal{T},w} = \delta_{i,j} \delta_{i,h} w_h = \delta_{j,h} w_h ,$$
whence $ \eta_h := \omega_h / \sqrt{w_h} $ is an orthonormal basis for $ \spaceV (E) $. Plugging this into   \eqref{eq:defconstantM}, one finds
\begin{align*}
    \sup_{T\in \averaging(E)} \sum_{i=1}^N \left| \sum_{h=1}^N T_i(\eta_h)T(\eta_h)\right|w_i & = \sup_{T\in \averaging(E)} \sum_{i=1}^N \left| \sum_{h=1}^N \frac{1}{w_h} \delta_{i,h} T(\omega_h) \right| w_i \\ & = \sup_{T\in \averaging(E)} \sum_{i=1}^N \left| T(\omega_h) \right| ,
\end{align*}
which is   \eqref{eq:defconstantL}, and represents the counterpart of the more frequent Lebesgue constant associated with an interpolation operator. The quantity introduced in   \eqref{eq:defconstantM} is indeed a very natural generalization of objects that are widely diffused in the literature, from the classical Lebesgue constant in nodal interpolation \cite{BrutmanSINUM} and fitting \cite{Reichel} to the generalized Lebesgue constant proposed in the simplicial case for Whitney forms \cite{AlonsoRapettiLebesgue} and histopolation \cite{BruniErb23}. We now scrutinize such relationships.

Suppose first that $ k = 0 $, so that elements of $ \averaging (E) $ are of the form $ T (\omega) = \frac{1}{L} \sum_{l=1}^L \omega (\xi_l) \tau_l $, where $ \xi_l \in E $ and $ \tau_l = \pm 1 $ for any $ l = 1, \ldots, L $. In the interpolatory case, due to the obvious set inclusion of nodal evaluations in $ \mathscr{A}^0 (E) $, one has
$$ \L(\mathcal T,\mathscr V^0(E)) = \sup_{T\in \mathscr{A}^0 (E)} \sum_{i=1}^N \left|  T(\omega_i)\right| \geq \sup_{\xi \in E} \sum_{i=1}^N | \omega_i (\xi) | = \max_{\xi \in E} \sum_{i=1}^N | \omega_i (\xi) | =: \mathrm{Leb}^0 (E) ,$$

In the above equation, $ \mathrm{Leb}^0 (E) $ is the usual nodal Lebesgue constant, i.e. the (supremum of) the sum of the modulus of the cardinal functions. On the other hand, for any $\varepsilon>0$, we can pick $T^\varepsilon\in \mathscr A^0(E)$, with $ T^\varepsilon (\omega) = \frac{1}{L} \sum_{l=1}^L \omega (\xi_l^\varepsilon) \tau_l^\varepsilon$, such that $\L(\mathcal T,\mathscr V^0(E))-\varepsilon\leq \sum_{i=1}^N|T^\varepsilon(\omega_i)|$. We thus obtain the converse inequality
$$ \L(\mathcal T,\mathscr V^0(E))-\varepsilon\leq \sum_{i=1}^N \left|\frac{1}{L} \sum_{l=1}^L \omega (\xi_l^\varepsilon) \tau_l^\varepsilon \right|\leq \sum_{l=1}^L \frac{1}{L}\sum_{i=1}^N \left|\omega_i(\xi^\varepsilon_l) \right|\leq \mathrm{Leb}^0 (E) .$$
This shows that, when $ k = 0 $, the quantity $ \L(\mathcal T,\mathscr V^0(E)) $ defined in \eqref{eq:defconstantL} is in fact the nodal Lebesgue constant $ \mathrm{Leb}^0 (E) $ used in classical interpolation of functions.
In the fitting case, when performing least squares approximation with respect the countable subset $\mathcal T=\{T_i\}_{i\in \N}$ of $\mathscr A^0(E)$ and the positive summable sequence $\{w_i\}_{i\in \N}$, the same relationship holds. The usual Lebesgue constant for least squares fitting  of the functionals $T_i := \omega (\xi_i) $'s with weights $w_i$'s is given by
$$ \mathrm{Leb}^0 (E) := \max_{\xi \in E} \sum_{i=1}^{+\infty} \left| \sum_{h=1}^N \eta_h (\xi_i) \eta_h (\xi) \right| w_i = \max_{\xi\in E}\sum_{i=1}^{+\infty} \left|\sum_{h=1}^NT_i(\eta_h)\eta_h(\xi) \right|w_i\,,$$
where $\{\eta_1,\dots,\eta_N\}$ is any orthonormal basis of $\mathscr V^0(E)$ with respect to $(\cdot,\cdot)_{\mathcal T,w}.$ Applying the same argument as above, one can replace in the definition of $\mathrm{Leb}^0 (E)$ the maximum over point evaluations by the supremum over $\mathscr A^0(E)$. Hence, for $k=0$, our definition of Lebesgue constant given   \eqref{eq:defconstantM} coincides with $\mathrm{Leb}^0 (E) $ also for the fitting case.

Moving to the case $ k > 0 $, in \cite{AlonsoRapettiLebesgue}\revised{ the }authors propose the quantity
$$ \mathrm{Leb}^k (E) := \sup_{S \in \mathcal{C}^k (E)} \frac{1}{\haus^k(S)}\sum_{i=1}^N \haus^k(S_i) \left| \int_S \widetilde{\omega}_i \right| ,$$
where $ \mathcal{C}^k (E) $ is the set of simplicial $k$-chains supported in $ E $, i.e. formal finite sums of $k$-simplices, $ \{ S_1, \ldots, S_N \} $ is a collection of oriented simplices and $ \{ \widetilde{\omega}_1, \ldots, \widetilde{\omega}_N \} $ satisfy the duality relationship $ \int_{S_i} \widetilde{\omega}_j = \delta_{i,j} $. Again, the Lebesgue constant proposed in   \eqref{eq:defconstantL} is consistent with $ \mathrm{Leb}^k (E) $:
\begin{align*}
        \L ( \mathcal{T}, \spaceV (E) ) & = \sup_{T\in \averaging(E)}\sum_{i=1}^N|T(\omega_i)| \geq  
 \sup_{S \in \mathcal{C}^k (E)} \sum_{i=1}^N \left| T_S (\omega_i) \right| \\ &
 = \sup_{S \in \mathcal{C}^k (E)} \frac{1}{\haus^k(S)}\sum_{i=1}^N \left| [S,\tau] (\omega_i) \right| \\ & = \sup_{S \in \mathcal{C}^k (E)} \frac{1}{\haus^k(S)}\sum_{i=1}^N \haus^k(S_i) \left| [S,\tau] \left(\frac{\omega_i}{\haus^k(S_i)}\right) \right| \\ &
        = \sup_{S \in \mathcal{C}^k (E)} \frac{1}{\haus^k(S)}\sum_{i=1}^N \haus^k(S_i) \left| \int_S \widetilde{\omega}_i \right| .
    \end{align*}
The converse inequality depends on the characterization of the norm $\Vert \cdot \Vert_0 $ pointed out in Remark \ref{rmk:simplicialapproximation} combined with   \eqref{normequivalence2}. In particular, under the hypothesis of the present work, the quantity $ \L ( \mathcal{T}, \spaceV (E) ) $ introduced in   \eqref{eq:defconstantL} and the quantity $ \mathrm{Leb}^k (E) $ defined in \cite{AlonsoRapettiLebesgue} coincide.

\subsection{Pushforwarding currents and projection operators} \label{sect:pushforward}

\revised{In order to work on a reference element and then extend results on \textquotedblleft physical\textquotedblright\ elements, it is fundamental to understand how the projection operators behave under the action of a sufficiently regular mapping. This is dual to understanding the \emph{pushforward} of the corresponding currents; of course, this depends on the specific class of linear functionals. From this perspective, integral and averaging currents have a remarkably different behavior, being only the former class invariant under the action of a diffeomorphism.}

Recall that any $\mathscr C^1$-map $ \varphi: \widehat{E}\rightarrow E := \varphi (\widehat E) $ induces a \emph{pullback} $ \varphi^* : \testforms (E)\rightarrow  \testforms ( \widehat{E} ) $, which acts as 
$$ \langle \varphi^* \omega(\widehat x); \widehat \tau_1\wedge\dots\wedge \widehat \tau_k\rangle = \langle\omega(\varphi(\widehat x); D\varphi(\widehat x) \widehat\tau_1\wedge\dots\wedge D\varphi(\widehat x) \widehat \tau_k\rangle, \;\;\forall \widehat\tau_1,\dots,\widehat\tau_k\in \R^n\,. $$
Then the \emph{pushforward} of currents in $(\testforms(\widehat E))'$ is defined similarly (see e.g., \cite[\S 7.4.2]{Krantz}) by setting
\begin{equation*}
\varphi_*\widehat T(\omega)=\widehat T(\varphi^*\omega),\;\;\forall \omega\in \testforms(E)\,.
\end{equation*}

Given a $N$-dimensional subspace $\spaceV(\widehat E)$ we will consider the following subspace of $\testforms(E)$: 
$$\spaceV(E):=\{\omega\in \testforms(E):\,\varphi^*\omega\in \spaceV(\widehat E)\}\,.$$
Note that, if $\varphi=\psi^{-1}$ is a $\mathscr C^1$-diffeomorphism onto its image, then we can equivalently characterize $\spaceV(E)$ as $\psi^*\spaceV(\widehat E)$, i.e.,  the image of $\spaceV(\widehat E)$ under $\psi^*$.
From now on we will always assume such an hypothesis, hence we will also have
$$\dim \spaceV (E) = \dim \spaceV (\widehat E) =  N.$$
\revised{Let $\widehat{\mathcal T}:=\{\widehat T_i\}_{i=1}^{+\infty}$ be a $\spaceV(\widehat E)$-determining subset of $\averaging(\widehat E)$, and let $\widehat w\in \ell^1(\N)$ be a positive sequence, so that $(\spaceV(\widehat E),(\cdot,\cdot)_{\widehat{\mathcal T},\widehat w})$ is a finite dimensional Hilbert space (see   \eqref{scalarproduct} above)}. Since we constructed $\spaceV(E)$ as an isomorphic copy of $\spaceV(\widehat E)$, it is rather natural to endow $\spaceV(E)$ with a scalar product that turns the above mentioned isomorphism into an isometry. Namely we set 
\begin{equation}\label{scalarproductpushforward}
(\omega,\eta)_{\mathcal T, w}:=(\varphi^*\omega,\varphi^*\eta)_{\widehat{\mathcal T},\widehat w}\,\;\;\forall \omega,\eta\in \testforms(E)\,.
\end{equation}
We warn the reader that this definition is not characterizing a unique choice of $\mathcal T\subset \averaging(E)$ and \revised{summable } positive weights $w$ such that the notation used in the left hand side of \eqref{scalarproductpushforward} is consistent with the one of \eqref{scalarproduct}. However, we show in   \eqref{newproduct} below that there exists a very natural choice of such $\mathcal T\subset \averaging(E)$ and \revised{summable } positive weights $w$, and we take that as a definition. We first show that the classes $\mathscr I^k$ and $\averaging$ behave slightly differently under the pushforward operation, and postpone the explicit computation of $\mathcal T$ and $w$ to equation \eqref{newproduct} below.

\begin{remark}[Averaging currents are not pushforward invariant] \label{rmk:pullbackaveraged}
If $E,\widehat E,$ $\varphi$ and $\psi$ are as above, and $\widehat S\subset \widehat E$ is a $k$-rectifiable set of finite $\haus^k$ measure and orientation $ \widehat \tau $, then $ S:=\varphi(\widehat S) $ is $k$-rectifiable, it has finite $\haus^k$ measure, and inherits the pushforward orientation $ \tau := D \varphi \widehat \tau / \Vert D \varphi \widehat \tau \Vert_{\Lambda_k} $ from $ \widehat S $ via $ \varphi $ ($ \Vert D \varphi \widehat \tau \Vert_{\Lambda_k} $ denotes the norm $ | \cdot |_{\Lambda_k}$ point-wisely considered). Thus one can write
\begin{equation*} \label{eq:pullback}
    [S,\tau](\omega)=[\varphi(\widehat S),\tau](\omega)=\int_{\varphi (\widehat S)} \omega = \int_{\widehat S} \varphi^* \omega=[\widehat S,\widehat \tau](\varphi^*\omega)=\varphi_*[\widehat S, \widehat \tau](\omega) \quad \forall \omega \in \testforms( E)\,.
\end{equation*}
Hence $\varphi_*:\mathscr I^k(\widehat E)\rightarrow \mathscr I^k(E)$, i.e., the class $\mathscr I^k$ is stable under pushforward. 

In contrast, applying this to the averaging current $T_{\widehat S}\in \averaging(\widehat E)$, one finds 
$$ \varphi_* T_{\widehat S} (\omega)=T_{\widehat S}(\varphi^*\omega)=\frac{[\widehat S, \widehat \tau](\varphi^*\omega)}{\haus^k(\widehat S)}=\frac{[\varphi (\widehat S),\tau](\omega)}{\haus^k(\widehat S)}=\frac{\haus^k(S)}{\haus^k(\widehat S)}T_{S}(\omega)\,.$$
Hence $\varphi_*T_{\widehat S}\notin \averaging(E)$, unless $\haus^k(\varphi(S))=\haus^k(S).$
\end{remark}
Now, if $\widehat{\mathcal T}:=\{\widehat T_i \}_{i=1}^{+\infty}$ is a subset of $\averaging (\widehat E)$, then for every $ i \in \N $, there exist the $k$-rectifiable set $ \widehat S_i $ such that $\widehat T_i=[\widehat S_i, \widehat \tau_i](\haus^k(\widehat S_i))^{-1}$, where $ \widehat \tau_i $ is a suitable orientation. For any $ i \in \N$, we denote by $S_i$ the set $\varphi(\widehat S_i)$ (with the inherited orientation $ \tau_i $) and we compute
\begin{align*}
(\omega,\eta)_{\mathcal T, w}:=&(\varphi^*\omega,\varphi^*\eta)_{\widehat{\mathcal T},\widehat w}=\sum_{i=1}^{+\infty}\widehat w_i \widehat T_i(\varphi^* \omega)\widehat T_i(\varphi^* \eta)\sum_{i=1}^{+\infty} \widehat w_i  \varphi_*\widehat T_i(\omega)\varphi_*\widehat T_i( \eta)\\
=& \sum_{i=1}^{+\infty} \widehat w_i \left(\frac{\haus^k(S_i)}{\haus^k(\widehat S_i)}\right)^2 T_{S_i}(\omega)T_{S_i}( \eta)\,.
\end{align*}
Hence, if we set 
\begin{equation}\label{newproduct}
T_i:=\frac{\haus^k(\widehat S_i)}{\haus^k( S_i)}\varphi_* \widehat T_i,\;\;\;w_i:=\widehat w_i \left(\frac{\haus^k(S_i)}{\haus^k(\widehat S_i)}\right)^2\qquad \forall i\in \N \,,
\end{equation}
the notation we introduced in \eqref{scalarproductpushforward} is consistent with the one of \eqref{scalarproduct}.

If we consider the projection operator $P:\testforms(E)\rightarrow \spaceV(E)$ defined by \eqref{discreteleastsquares}, with $\mathcal T$ and $w$ as in \eqref{newproduct}, then we get by construction
\begin{equation*}
 P  \omega=\psi^*\widehat P\varphi^*\omega,\;\;\forall \omega\in \testforms( E)\,,
\end{equation*}
i.e., the diagram 
\begin{center}
\begin{tikzcd}
\testforms(\widehat{E}) \arrow[d, "\widehat P"] 
& \testforms(E) \arrow[d, "P" ] \arrow[l, "\varphi^*"] \\
 \spaceV(\widehat{E}) \arrow[r, "\psi^*"]
&  \spaceV(E)
\end{tikzcd}
\end{center}
commutes.

Clearly, we can give an equivalent operative definition of $P$ using orthonormal bases. Indeed, it follows from definition \eqref{scalarproductpushforward} that, given an orthonormal basis $\{\widehat\eta_1,\dots\widehat \eta_N\}$ of $(\spaceV(\widehat E),(\cdot,\cdot)_{\widehat{\mathcal T},\widehat w})$, the set
\begin{equation}\label{neworthogonalbasis}
\{\eta_1,\dots,\eta_N\}:=\{\psi^*\widehat\eta_1,\dots,\psi^*\widehat\eta_N\}
\end{equation}
forms an orthonormal basis of $(\spaceV(E),(\cdot,\cdot)_{\mathcal T,w})$, whence 
$ P\omega:=\sum_{h=1}^N(\omega,\eta_h)_{\mathcal T,w}\eta_h $, 
for $ \omega\in \testforms(E) $.

\begin{remark}
    In the interpolatory case $ \Card \mathcal T = N $ (hence $ M = \Card \mathcal T < + \infty $), 
$ \varphi_* $ preserves unisolvence as the pullback $ \varphi^* $ preserves cardinal bases associated with non-averaged currents. 
In fact, if $ \{ \widehat {\omega}_1, \ldots, \widehat{\omega}_N \} $ is the Lagrange basis associated with $ \{ \widehat{S}_1, \ldots, \widehat{S}_N \} $, 
one has:
$$ \delta_{i,j} = [\widehat S_i, \widehat \tau_i] \widehat \omega_j = \int_{\widehat {S}_i} \widehat \omega_j = \int_{\psi (S_i)} \widehat \omega_j = \int_{S_i} \psi^* \widehat \omega_j = [S_i, \tau_i] \psi^* \widehat \omega_j =: [S_i, \tau_i] \omega_j \,,$$
being $ \tau_i := \varphi_* \widehat \tau_i / \Vert \varphi_* \widehat \tau_i \Vert_{\Lambda_k} $, see Remark \ref{rmk:pullbackaveraged}. In other words, the basis $ \{ \omega_1, \ldots, \omega_N \} $ is dual to the pushforward currents. In the averaged case, the above chain of equalities remains consistently true provided that the normalization factor of   \eqref{newproduct} is introduced.
%
%
\end{remark}


\begin{remark}
It is worth pointing out here that, although we introduced $(\spaceV(E),(\cdot,\cdot)_{\mathcal T,w})$ declaring it isometric to $(\spaceV(\widehat E),(\cdot,\cdot)_{\widehat{\mathcal T},\widehat w})$, this relation is not in general compatible with the Banach spaces $\testforms(E)$ and $\testforms(\widehat E)$ standing upon these finite dimensional subspaces. \revised{More in detail, exploiting a singular values decomposition, one sees that, for any $\omega\in \testforms(E)$ and any $\mathscr C^1$-diffeomorphism $\varphi: \widehat E\rightarrow E$, one has
\begin{equation}\label{eq:zeronormchange}
\frac 1{\Vert \prod_{j=n-k+1}^n (\sigma_j^{(\varphi)})^{-1}\Vert_{\mathscr C^0(\widehat E)}}\|\omega\|_0\leq \|\varphi^*\omega\|_0\leq \left\Vert \prod_{j=1}^k \sigma_j^{(\varphi)}\right\Vert_{\mathscr C^0(\widehat E)} \|\omega\|_0,
\end{equation}
where we denoted by $\sigma_j^{(\varphi)}:\widehat E\rightarrow (0,+\infty)$ the singular value functions of the differential $D\varphi$ of $\varphi$ arranged in a non-decreasing order, and by $\|\cdot\|_{\mathscr C^0(\widehat E)}$ the uniform norm of a continuous function over $\widehat E$. We here avoid a formal proof for this result, which is rather long and technical, and will follow from Section \ref{sect:changing} as an application of Lemma \ref{lem:fede}. Note also that \eqref{eq:zeronormchange} is sharp in the sense that, for any small enough $\varepsilon>0$, we can pick $\omega_\varepsilon\in \testforms(E)$ in order to make the right inequality hold as an equality up to $\varepsilon$, i.e. to obtain
\begin{align*}
&\left(\left\Vert \prod_{j=1}^k \sigma_j^{(\varphi)}\right\Vert_{\mathscr C^0(\widehat E)}-\varepsilon\right) \|\omega_\varepsilon\|_0\leq\|\varphi^*\omega_\varepsilon\|_0\leq \left\Vert \prod_{j=1}^k \sigma_j^{(\varphi)}\right\Vert_{\mathscr C^0(\widehat E)} \|\omega_\varepsilon\|_0\,.
\end{align*}
A similar choice can be made for the left inequality in \eqref{eq:zeronormchange}. }
\end{remark}

%

\section{Operator norm and Lebesgue constant}\label{sec:3}

\revised{The principal aim of this section is to prove the first main result of the paper:

\noindent \textbf{Theorem $\boldsymbol{1}$} (Characterization of the norm of the projection operator)\textbf{.}
\emph{Let $\mathcal T=\{T_1,T_2,\dots\}$  be a countable $\spaceV(E)$-determining subset of $\averaging(E)$ and $w\in \ell^1(\N)$. If $\haus^k(\support T_i \cap \support T_j)=0$ for $i\neq j$, then the discrete least squares projection $P$ satisfies
\begin{equation*}
\opnorm{P}= \L(\mathcal T,\spaceV (E))\,.
\end{equation*}
Inequality  $\opnorm{P}\leq\L(\mathcal T,\spaceV (E))$ still holds even if $\haus^k(\support T_i \cap \support T_j)>0$ for some $i\neq j$.}}

The proof of this result is carried in the following subsection, whereas its consequences are scrutinized in the subsequent Section \ref{sec:contL} and Section \ref{sec:BanachHilbert}. \revised{For ease of notation, in the rest of the section we will omit the dependence of $ \L $ on $ \mathcal{T} $ and $ \spaceV (E) $ when clarified by the context.}

\subsection{Proof of Theorem \ref{thm1}}
We build the proof of Theorem \ref{thm1} by the following steps. In Lemma \ref{lemma:inequalitycase}, we prove that whenever $ \mathcal{T} = \{T_i\}_{i\in \N}$ is a determining set for $ \spaceV (E) $ and $\|w\|_{\ell^1}<+\infty$, the norm $ \Vert P \Vert_{\mathrm{op}} $ of the projector $ P : \testforms (E) \to \spaceV (E) $ is upper bounded by $ \L $: $ \Vert P \Vert_{\mathrm{op}} \leq \L $. The converse inequality is not always true \cite{BruniErb23}. We prove such a result assuming that intersections of the supports have vanishing $ \haus^k$-measure. To this aim, we first prove that $ \Vert P \Vert_{\mathrm{op}} \geq \L $ holds under two additional assumptions on $\mathcal T$, see Lemma \ref{lemma:equalitycase}. Then the general case is achieved by an approximation argument. This technique relies on the construction of suitable approximations $\mathcal T^{(\varepsilon)}$ of $\mathcal T$ that on one hand satisfy the assumptions of Lemma \ref{lemma:equalitycase}, and on the other hand are uniformly converging to $\mathcal T$ as linear operators from $\testforms$ to $\ell^1_w$ (this concept of convergence is expanded in \eqref{eq:uniformconvergenceT}). Exploiting these properties, we prove in Lemma \ref{lemma:normconvergence} the strong convergence of the corresponding projectors $ P^{(\varepsilon)} $ towards $ P $ and, in Lemma \ref{lemma:lebesguesemicontinuity}, the lower semicontinuity of the associated Lebesgue constants $ \L_\varepsilon := \L (\mathcal{T}^{(\varepsilon)}, \spaceV (E) )$. This yields the claim of Theorem \ref{thm1}, since
\begin{equation}\label{eq:toproveth1}
 \L \leq \liminf_{\varepsilon \to 0} \L_\varepsilon = \liminf_{\varepsilon \to 0} \Vert P^{(\varepsilon)} \Vert_{\mathrm{op}} = \lim_{\varepsilon \to 0} \Vert P^{(\varepsilon)} \Vert_{\mathrm{op}} = \Vert P \Vert_{\mathrm{op}} \leq \L .
\end{equation}

\begin{lemma}\label{lemma:inequalitycase}
Let $\mathcal T=\{T_i\}_{i\in \N}\subset \averaging(E)$ be a $\spaceV (E) $-determining set and $w\in\ell^1 (\N)$ be a positive sequence. The discrete least squares projection $P$  
satisfies
\begin{equation}\label{Pnormestimaterepeated}
\opnorm{P}\leq\L\,.
\end{equation}
\end{lemma}
\begin{proof}
Using the representations of $P$ and $\L$ based on scalar product (respectively \eqref{eq:Pbyscalarproduct} and   \eqref{eq:rephrasingL}), we can write
\begin{align*}
\|P\omega\|_0\leq & \|\{T_j(\omega)\}_{j\in \N}\|_{\ell^\infty}\sup_{T\in \averaging(E)}\sum_{i=1}^{+\infty}|T_i(K_T)|w_i\leq \|\omega\|_0 \sup_{T\in \averaging(E)}\sum_{i=1}^M|T_i(K_T)|w_i\\
=&\|\omega\|_0\sup_{T\in \averaging(E)}\sum_{i=1}^{+\infty}\left|\sum_{h=1}^NT_i(\eta_h)T(\eta_h)\right|w_i=\|\omega\|_0\L\,.
\end{align*}
  The claimed inequality \eqref{Pnormestimaterepeated} follows by taking the supremum over $\omega\in \testforms(E)$ with $\|\omega\|_0=1$.
\end{proof}

The distance between $ \Vert \Pi \Vert_{\mathrm{op}} $ and $ \L $ has been observed (in the interpolatory case) \revised{when the supports $\{ S_i \}_{i=1}^N $ of the currents $ \{ T_i \}_{i=1}^N $ } present a significant overlapping, see \cite[Figure 3]{BruniErb23}. In the opposite scenario, i.e. when there is some distance between the supports, 
one may exploit bump forms \cite[p. 25]{BottTu} and construct $ \omega \in \testforms (E) $ with $ \Vert \omega \Vert_0 = 1 $ for which equality in the above proof is attained. The intermediate case  $\haus^k(S_i\cap S_j)=0$ for all $i \ne j$ is more involved. In order to prove that equality holds under the natural assumption of disjointedness in measure of the supports, we first establish the same result under much stronger assumptions.
\begin{lemma}\label{lemma:equalitycase}
Assume that the $\spaceV(E)$-determining set $\mathcal T$ fulfills the following properties:
\begin{enumerate}[i)]
\item $\mathcal T=\{T_1,\dots,T_M\}$ with $M<+\infty\,;$
\item the sets $S_i:=\support T_i$ are of the form $S_i=\cup_{l=1}^{\ell_i} S_{i,l}$, $\ell_i<+\infty$, where the sets $S_{i,l}$ are compact subsets of $\mathscr C^1$ embedded oriented $k$-submanifolds of $\R^n$, with orientation $\tau^{i,l}$, and $S_{i,l}\cap S_{j,m}=\emptyset$ whenever $i\neq j$ or $l\neq m.$ 
\end{enumerate}
Then
\begin{equation}\label{Pnormreverseestimate}
 \Vert P \Vert_{\mathrm{op}} \geq \L .
\end{equation}
Hence $\Vert P \Vert_{\mathrm{op}} = \L$.
\end{lemma}

\begin{proof}
Due to   \eqref{eq:rephrasingL}, setting $s_i(T):=\sgn(T_i(K_T))$, we rephrase the Lebesgue constant $ \L $ using the \Riz representer $K_T$ of $T$
\begin{equation}\label{eq:redefineL}
\L=\sup_{T\in \averaging(E)}\sum_{i=1}^M\vert T_i(K_T)\vert w_i=\sup_{T\in \averaging(E)}\sum_{i=1}^Ms_i(T) T_i(K_T) w_i\,.
\end{equation}
Since $\{S_{i,l},i=1,\dots,M,\; l=1,\dots,\ell_i\}$ is a finite collection of disjoint compact sets, we can pick a family of disjoint open neighborhoods of the elements of such a set. Namely
$$U_{i,l}\supset S_{i,l},\;U_{i,l}\text{ open, }U_{i,l}\cap U_{j,m}=\emptyset,\text{ whenever }i\neq j\text{ or }l\neq m\,.$$
For any $i=1,\dots,M$, $l = 1, \ldots, \ell_i$, we pick a bump form $\omega_{i,l}^T \in \testforms (\R^n) $ compactly supported in $ U_{i,l}$ such that
\begin{align*}
&\omega_{i,l}^T(x)= s_i(\tau_1^{i,l})^*\wedge\dots\wedge(\tau_k^{i,l})^*,\;\;\forall x\in S_{i,l}\\
&\|\omega_{i,l}^T\|_0=1\\
&\sgn \langle\omega_{i,l}^T(x);\tau^{i,l}(x)\rangle=s_i \,,
\end{align*}
where $ \tau^{i,l} $ is the orientation of $ S_{i,l} $ and $(\cdot)^*$ denotes the standard duality between vectors and covectors. Then define the form $\omega^T\in \testforms(E)$ by gluing the forms $\omega^T_{i,l}$ restricted to $E$, i.e.
$$\omega^T(x):=\begin{cases}
\omega_{i,l}^T(x)& x\in U_{i,l}\cap E, \\
0& \text{ otherwise.}
\end{cases}
$$
Notice that $\omega^T$ has been constructed in order to simultaneously get $T_i(\omega^T)=s_i$, $i=1,\dots,M$, and $\|\omega^T\|_0=1$. Finally compute
\begin{align*}
\opnorm{P}=&\sup_{\|\omega\|_0=1}\|P\omega\|_0=\sup_{\|\omega\|_0=1}\sup_{T\in \averaging(E)} (K_T,\omega)_{\mathcal T,w}=\sup_{T\in \averaging(E)}\sup_{\|\omega\|_0=1} (K_T,\omega)_{\mathcal T,w}\\
\geq & \sup_{T\in \averaging(E)} (K_T,\omega^T)_{\mathcal T,w} =\sup_{T\in \averaging(E)}\sum_{i=1}^MT_i(K_T)T_i(\omega^T) w_i=\sup_{T\in \averaging(E)}\sum_{i=1}^Ms_iT_i(\omega^T) w_i\\
=&\L\,,
\end{align*}
where the last equality is due to \eqref{eq:redefineL}.
\end{proof}

\revised{
In order to extend the above result to the case of honest averaging currents we need to introduce an approximation procedure and some new notation. The idea behind our construction is to approximate the countable set of currents $\mathcal T$ of Theorem \ref{thm1} by a set of currents $\mathcal T^{(\varepsilon)}=\{T^{(\varepsilon)}_1,\dots,T^{(\varepsilon)}_{M(\varepsilon)},0,\dots\}$ having mutually disjoint supports, each of whose is made of a disjoint union of a finite collection of compact subsets of $\mathscr C^1$ embedded manifolds with continuous orientations. Notice that, in view of Lemma \ref{lemma:equalitycase}, this in particular implies 
$$ \Vert P^{(\varepsilon)} \Vert_{\mathrm{op}} = \L (\mathcal{T}^{(\varepsilon)}, \spaceV (E) ) =: \L_\varepsilon ,$$
where $ P^{(\varepsilon)} $ is the projection operator onto $ \spaceV (E) $ associated with $ \mathcal{T}^{(\varepsilon)} $ and weights $ w $. In fact, the ultimate aim of this procedure is to obtain the continuity of the norm of the induced fitting operators, i.e. $\varepsilon\to \opnorm{P^{(\varepsilon)}}$, and the lower semicontinuity of the Lebesgue constants $\L_{\varepsilon}.$ These properties are the ingredients used in \eqref{eq:toproveth1} to conclude the proof of Theorem \ref{thm1}.

Consider
$$\mathcal T:=\{T_i\}_{i=1}^{+\infty}\,,\;\;T_i:=\frac{[S_i, \tau_i]}{\haus^k(S_i)},\;\;S_i:=\bigcup_{l=0}^{+\infty} S_{i,l}\,,$$
where, for every $i\in \N$, $\haus^k(S_{i,0})=0$, and, for any $l>0$, the sets $S_{i,l}$ are measurable subsets of embedded $\mathscr C^1$ submanifolds of $\R^n$, with $S_{i,l}\cap S_{i,m}=\emptyset$ whenever $l\neq m$. 
If we further assume, as in Theorem \ref{thm1}, that $\haus^k(S_i\cap S_j)=0$ for any pair of indices $i\neq j$, then, for any $i\in \N$, the set 
$$\widetilde S_i:=S_i\setminus[S_{i,0}\bigcup(\cup_{j\neq i}(S_i\cap S_j))]$$
is a carrier for $T_i$. Indeed, we have
$$T_i=\frac{[\widetilde S_i,\tau_i]}{\haus^k(S_i)}\,,\;\forall i\in \N\,.$$
For each $ i $, the set $\widetilde S_i$ is still a $k$-rectifiable set of finite measure oriented by $\tau_i$. 

The orientation $\tau_i$ is, by definition, an $\haus^k$-measurable function defined on a finite measure space. Therefore, by Lusin Theorem \cite[Theorem 7.10]{Folland}, for any $\varepsilon>0$ and $i\in \N$, we can find a closed subset $\widetilde S_i^{(\varepsilon)}\subset \widetilde S_i$ such that $\tau_i|_{\widetilde S_i^{(\varepsilon)}}$ is continuous, and 
\begin{equation}\label{massleftout}
\haus^k(\widetilde S_i\setminus \widetilde S_i^{(\varepsilon)})\leq \frac \varepsilon 2 \min\{1,\haus^k(S_i)\}.
\end{equation}
Again, the set $\widetilde S_i^{(\varepsilon)}$ is still $k$-rectifiable, hence we can represent it as
$$\widetilde S_i^{(\varepsilon)}=\bigcup_{l=0}^{+\infty} \widetilde S_{i,l}^{(\varepsilon)}
,$$
where, for any $i\in \N$, the set $\widetilde S_{i,0}^{(\varepsilon)}$ has zero $ \haus^k $-measure, $\widetilde S_{i,l}^{(\varepsilon)}\cap \widetilde S_{i,m}^{(\varepsilon)}=\emptyset$ whenever $l\neq m$, and, for $l>0$, $\widetilde S_{i,l}^{(\varepsilon)}$, is a measurable piece of a $\mathscr C^1$ embedded $k$-submanifold of $\R^n$.

For any $\varepsilon>0$, let 
\begin{align*}
M(\varepsilon):=&\min\left\{M:\sum_{i=M+1}^{+\infty} w_i < \varepsilon \sum_{i=1}^{+\infty}w_i  \right\},\\
\ell_i(\varepsilon):=&\min\left\{L:\sum_{l=L+1}^{+\infty}\haus^k(\widetilde S_{i,l}^{(\varepsilon)}) <\frac \varepsilon 4 \min \{1,\haus^k(S_i) \}  \right\}, \quad i=1,2,\dots,M(\varepsilon).
\end{align*}
Exploiting the inner regularity of Hausdorff measure on sets of finite measure (see \cite[Theorem 1.6(b)]{Fa86}), for any $i\in \{1,2,\dots, M(\varepsilon)\}$, we can pick compact subsets $C_{i,l}(\varepsilon)$ of $\widetilde S_{i,l}^{(\varepsilon)}$ such that 
\begin{equation*}
\haus^k(\widetilde S_{i,l}^{(\varepsilon)}\setminus C_{i,l}(\varepsilon))\leq \frac \varepsilon {4\ell_i(\varepsilon)}\min\{1,\haus^k(S_i)\}\,.
\end{equation*}

Finally, set
\begin{equation}\label{eq:myapprox}
S_i^{(\varepsilon)}:= \bigcup_{l=1}^{\ell_i(\varepsilon)} C_{i,l}(\varepsilon)\,,\;\;T_i^{(\varepsilon)}:=\begin{cases}
\frac{[S_i^{(\varepsilon)},\tau_i]}{\haus^k(S_i^{(\varepsilon)})} &1\leq i\leq M(\varepsilon)\\
0&\text{ otherwise }\end{cases}\,,\;\;\mathcal T^{(\varepsilon)}:=\{T_i^{(\varepsilon)}\}_{i=1}^{+\infty}\,.
\end{equation}
Let us collect some of the properties of this approximation sequence:

\begin{lemma}\label{lem:preparation}
Let $\mathcal T$ be as above and, for any $\varepsilon>0$, let $\mathcal T^{(\varepsilon)}$ be defined as in \eqref{eq:myapprox}. Then the following hold:
\begin{enumerate}[a)]
\item $\mathcal T^{(\varepsilon)}$ converges uniformly to $\mathcal T$, i.e.,
\begin{equation}\label{eq:uniformconvergenceT}
\lim_{\varepsilon\to 0^+} \sup_{\Vert \omega \Vert_0 = 1} \Vert [\mathcal{T} - \mathcal{T}^{(\varepsilon)} ] (\omega)\Vert_{\ell^1_w}=0\,;
\end{equation}
\item The reproducing kernel $K^{(\varepsilon)}$ of $\Big(\spaceV (E),(\cdot,\cdot)_{\mathcal T^{(\varepsilon)},w}\Big)$ converges uniformly to the reproducing kernel $ K $ of $\Big(\spaceV (E),(\cdot,\cdot)_{\mathcal T,w}\Big)$, i.e.,
\begin{equation}\label{eq:convergenceofkernelI}
\lim_{\varepsilon\to 0} \|K -K^{(\varepsilon)}\|_0 = 0\,,
\end{equation}
where, for any $\omega,\eta\in \testforms(E)$, we set $ \Vert \omega \otimes \eta \Vert_0 := \sup_{S,T \in \averaging(E)} | (S \otimes T) (\omega \otimes \eta) | ;$
\end{enumerate}
\end{lemma}
 
\begin{proof}
Note that, by construction and due to \eqref{massleftout}, for all $i=1,\dots,M(\varepsilon)$, 
$$ \Vert [ S_i, \tau_i ] - [ S_i^{(\varepsilon)}, \tau_i ] \Vert_{\mathrm{op}}\leq\haus^k(S_i\setminus \widetilde S_i^{(\varepsilon)})+\haus^k( \widetilde S_i^{(\varepsilon)}\setminus S_i^{(\varepsilon)}) \leq \frac \varepsilon 2 \min \{ 1,\haus^k(S_i) \}\,.$$

As a consequence, by a polarization argument and the triangular inequality, one has, for any such $i$, 
\begin{equation} \label{eq:estimateR}
    \Vert T_i - T_i^{(\varepsilon)} \Vert_{\mathrm{op}} \leq 2\frac{\haus^k(S_i\setminus S_i^{(\varepsilon)})}{\haus^k(S_i)}\leq \varepsilon\frac{\min \{1,\haus^k(S_i)\}}{\haus^k(S_i)}\leq \varepsilon\,.
\end{equation}
The uniform convergence of $ \mathcal{T}^{(\varepsilon)}$ to $ \mathcal{T} $ immediately follows. In fact, one has
\begin{align*} 
&\sup_{\Vert \omega \Vert_0 = 1} \Vert [\mathcal{T} - \mathcal{T}^{(\varepsilon)} ] (\omega)\Vert_{\ell^1_w} 
 =\sup_{\Vert \omega \Vert_0 = 1 } \left( \sum_{i=1}^{+\infty}|T_i(\omega)-T_i^{(\varepsilon)}(\omega)|w_i \right) \nonumber \\
 \leq& \sup_{\Vert \omega \Vert_0 = 1 } \left( \sum_{i=1}^{M(\varepsilon)}|T_i(\omega)-T_i^{(\varepsilon)}(\omega)|w_i \right) + \sum_{i=M(\varepsilon) + 1}^{+ \infty} w_i \leq 2 \varepsilon \Vert w \Vert_{\ell^1},
\end{align*}
which tends to zero as $ \varepsilon \to 0 $, since the weights are summable.

Let $\{\vartheta_h\}_{h=1}^N $ be any basis of $ \spaceV (E) $ and let $ \mathcal{A} $ be any continuous matrix orthogonalization algorithm (such as, e.g., Gram-Schmidt). Denote by $ \{ \eta_h \}_{h=1}^N $ the orthonormal basis of $ \spaceV (E) $ with respect to $(\cdot,\cdot)_{\mathcal T, w} $ obtained from the Gram matrix $ G:=[(\vartheta_h,\vartheta_k)_{\mathcal T,w}]_{h,k=1}^N $ by $ \mathcal{A} $; likewise, denote by $ \{ \eta_h^{(\varepsilon)} \}_{h=1}^N $ the orthonormal basis of $\spaceV(E)$ with respect to $(\cdot,\cdot)_{\mathcal T^{(\varepsilon)}, w} $ obtained from $ G^{(\varepsilon)}:= [(\vartheta_h,\vartheta_k)_{\mathcal T^{(\varepsilon)},w}]_{h,k=1}^N  $ by $ \mathcal{A} $. 
One has
\begin{equation} \label{eq:estimateeta}
\max_{h=1,\dots ,N}\Vert \eta_h^{(\varepsilon)} - \eta_h \Vert_0 \to 0
\end{equation}
as $ \varepsilon \to 0 $. Indeed, due to the continuity of $\mathcal A$, it is sufficient to show that the entries of $ G^{(\varepsilon)} $ are uniformly converging to those of $ G $. 
Using \eqref{eq:estimateR}, we compute
\begin{align*}
&\max_{h,k}|(\vartheta_h,\vartheta_k)_{\mathcal T,w}-(\vartheta_h,\vartheta_k)_{\mathcal T^{(\varepsilon)},w}|\\
\leq & \max_{h,k}\sum_{i=1}^{+\infty}|(T_i(\vartheta_h)-T_i^\varepsilon(\vartheta_h))T_i(\vartheta_k)w_i|+ \max_{h,k}\sum_{i=1}^{+ \infty}|T_i^\varepsilon(\vartheta_h)(T_i(\vartheta_k)-T_i^\varepsilon(\vartheta_k)))w_i|\\
\leq& \max_{h,k}\|\vartheta_h\|_0\|\vartheta_k\|_0 2 \sup_{\Vert \omega \Vert_0 = 1 } \left( \sum_{i=1}^{+\infty}|T_i(\omega)-T_i^{(\varepsilon)}(\omega)|w_i \right),
\end{align*}
which converges to zero as $ \varepsilon \to 0 $. As a by-product, \eqref{eq:estimateeta} yields the convergence of the reproducing kernels
\begin{equation*}
\lim_{\varepsilon\to 0} \|K -K^{(\varepsilon)}\|_0 = 0\,,
\end{equation*}
where $ K(x,y):=\sum_{h=1}^N\eta_h(x)\otimes \eta_h(y) $ and $ K^{(\varepsilon)}(x,y):=\sum_{h=1}^N\eta^{(\varepsilon)}_h(x)\otimes \eta^{(\varepsilon)}_h(y) $.
\end{proof}
}
Lemma \ref{lem:preparation} puts us in a position to prove the following.
\begin{lemma}\label{lemma:normconvergence}
    Under the above assumptions,
    $$ \Vert P - P^{(\varepsilon)} \Vert_{\mathrm{op}} \to 0 $$
    as $ \varepsilon \to 0 $. In particular, $ \Vert P^{(\varepsilon)} \Vert_{\mathrm{op}} \to \Vert P \Vert_{\mathrm{op}} $.
\end{lemma}

\begin{proof}
    Using the reproducing kernel property and adding and subtracting suitable terms, we find
    \begin{align*}
        [ P - P^{(\varepsilon)}] (\omega) & = (K, \omega)_{\mathcal{T},w} - (K^{(\varepsilon)}, \omega)_{\mathcal{T}^{(\varepsilon)},w} = \sum_{i=1}^{+\infty} \left( T_i (K) T_i (\omega) - T^{(\varepsilon)}_i (K^{(\varepsilon)}) T^{(\varepsilon)}_i (\omega) \right) w_i \\
        & = \sum_{i=1}^{+ \infty} \Bigg[ T_i(K) \left( T_i - T_i^{(\varepsilon)} \right) (\omega) + T_i^{(\varepsilon)} (\omega) T_i (K - K^{(\varepsilon)} ) +\\
        &\;\;\;\;\;\;\;\;\;\;\;\;\;\;\;\;\;\;\;\;\;\;\;\;\;\;\;\;\;\;\;\;\;\;\;\;\;\;\;\;\;\;\;\; T_i^{(\varepsilon)} (\omega)\left( T_i - T_i^{(\varepsilon)}\right) ( K^{(\varepsilon)}) \Bigg] w_i .
    \end{align*}
    Hence, passing to norms and applying the triangular inequality, we can bound $ \Vert [P - P^{(\varepsilon)}] (\omega) \Vert_0 $ from above as the sum of three terms:
    \begin{align*}
        \Vert P - P^{(\varepsilon)} \Vert_{\mathrm{op}} \leq& 
        \sup_{\Vert \omega \Vert_0 = 1} \sup_{T \in \averaging (E)}  \sum_{i=1}^{+ \infty} \left| (T\otimes T_i)(K) \left( T_i - T_i^{(\varepsilon)} \right) (\omega)\right|w_i \\
        &+ \sup_{\Vert \omega \Vert_0 = 1} \sup_{T \in \averaging (E)}  \sum_{i=1}^{+ \infty} \left|T_i^{(\varepsilon)} (\omega) (T\otimes T_i) (K - K^{(\varepsilon)} )\right|w_i \\
        &+ \sup_{\Vert \omega \Vert_0 = 1} \sup_{T \in \averaging (E)}  \sum_{i=1}^{+ \infty} \left| T_i^{(\varepsilon)} (\omega)\left(T\otimes \left( T_i - T_i^{(\varepsilon)}\right)\right) ( K^{(\varepsilon)}) \right| w_i .    \end{align*}
        We show that each one of the three terms in the above sum converges to zero as $ \varepsilon \to 0 $. The first one can be handled by using the H\"older inequality:
        \begin{align*} & \sup_{\Vert \omega \Vert_0 = 1} \sup_{T \in \averaging (E)} \sum_{i=1}^{+ \infty} \left| (T\otimes T_i)(K) \left( T_i - T_i^{(\varepsilon)} \right) (\omega)\right|w_i  \\ \leq & \sup_{T \in \averaging (E)} \Vert \left\{ [(T\otimes T_i) (K)] \right\} \Vert_{\ell^{\infty}} \sup_{\Vert \omega \Vert_0 = 1} \Vert [\mathcal{T} - \mathcal{T}^{(\varepsilon)} ] (\omega)\Vert_{\ell^1_w} \leq \Vert K \Vert_0 \sup_{\Vert \omega \Vert_0 = 1} \Vert [\mathcal{T} - \mathcal{T}^{(\varepsilon)} ] (\omega)\Vert_{\ell^1_w} ,
        \end{align*}
        which tends to zero as $ \varepsilon \to 0 $ due to   \eqref{eq:uniformconvergenceT}. The same argument, paired with   \eqref{eq:convergenceofkernelI}, yields the convergence of the second term. The third one requires a more careful approach:
        \begin{align*}
            & \sup_{\Vert \omega \Vert_0 = 1} \sup_{T \in \averaging (E)} \sum_{i=1}^{+ \infty} \left| T_i^{(\varepsilon)} (\omega)\left(T\otimes \left( T_i - T_i^{(\varepsilon)}\right)\right) ( K^{(\varepsilon)}) \right| w_i \\
            \leq & \sup_{T \in \averaging (E)} \sum_{i=1}^{+ \infty} \left|\left(T\otimes \left( T_i - T_i^{(\varepsilon)}\right)\right) ( K^{(\varepsilon)}) \right| w_i = \sup_{T \in \averaging (E)} \sum_{i=1}^{+ \infty} \left| \sum_{h=1}^N T(\eta_h^{(\varepsilon)}) (T_i - T_i^{(\varepsilon)}) (\eta_h^{(\varepsilon)}) \right| w_i \\
            \leq & \max_h \Vert \eta_h^{(\varepsilon)} \Vert_0 \sum_{i=1}^{+ \infty} \sum_{h=1}^N \left| (T_i - T_i^{(\varepsilon)}) (\eta_h^{(\varepsilon)}) \right| w_i = \max_h \Vert \eta_h^{(\varepsilon)} \Vert_0 \sum_{h=1}^N \Vert [\mathcal{T} - \mathcal{T}^{(\varepsilon)} ] (\eta_h^{(\varepsilon)})\Vert_{\ell^1_w} \\
            \leq & N (\max_h \Vert \eta_h^{(\varepsilon)} \Vert_0)^2 \sup_{\Vert \omega \Vert_0 = 1} \Vert [\mathcal{T} - \mathcal{T}^{(\varepsilon)} ] (\omega)\Vert_{\ell^1_w},
        \end{align*}
which again converges to zero as $ \varepsilon \to 0 $ due to   \eqref{eq:uniformconvergenceT}.
\end{proof}

To conclude the proof of Theorem \ref{thm1}, we need a last ingredient: the lower semicontinuity of the Lebesgue constant $\L_\varepsilon$, when regarded as a function of one real variable $\varepsilon\in(-\infty,1)$. Precisely, we set
\begin{equation*}
    L(\varepsilon) :=\begin{cases}
	\L& \varepsilon\in (-\infty,0]\\    
    \L_\varepsilon& \varepsilon\in(0,1)
    \end{cases}\;=\;
    \sup_{T\in \averaging(E)} \begin{cases}
    \sum_{i=1}^{+\infty}\left| \sum_{h=1}^NT_i (\eta_h)T(\eta_h)\right|w_i & \varepsilon\in(-\infty,0], \\
    \sum_{i=1}^{M(\varepsilon)}\left| \sum_{h=1}^NT_i^{(\varepsilon)} (\eta_h^{(\varepsilon)})T(\eta_h^{(\varepsilon)})\right|w_i & \varepsilon\in(0,1).
    \end{cases}
\end{equation*}

\begin{lemma}\label{lemma:lebesguesemicontinuity}
    The function $ \varepsilon \mapsto L(\varepsilon) $ is lower semicontinuous. In particular
    \begin{equation} \label{eq:Llowersemicontinuity}
    \L\leq \liminf_{\varepsilon\to 0^+}\L_\varepsilon.
    \end{equation}
\end{lemma}

\begin{proof}
By definition, $ L $ is the point-wise supremum of the family of functions
$$f_T(\varepsilon):= 
\begin{cases}
    \sum_{i=1}^{+\infty}\left| \sum_{h=1}^NT_i (\eta_h)T(\eta_h)\right|w_i & \varepsilon\in(-\infty,0], \\
    \sum_{i=1}^{M(\varepsilon)}\left| \sum_{h=1}^NT_i^{(\varepsilon)} (\eta_h^{(\varepsilon)})T(\eta_h^{(\varepsilon)})\right|w_i & \varepsilon\in(0,1) .
    \end{cases}
$$
The continuity of $ f_T $ at any $\varepsilon\in(-\infty,1)$ and for every $ T \in \averaging (E) $ easily follows from the strong convergence of $T_i^{\varepsilon+\delta}\to T_i^\varepsilon$ given in \eqref{eq:estimateR}, the norm convergence of $\eta_h^{\varepsilon+\delta}\to \eta_h^\varepsilon$ introduced in \eqref{eq:estimateeta}, and the summability of the weights $w_i$'s. Since $L$ is the pointwise supremum of a family of continuous functions, it is also lower semicontinuous.
\end{proof}

Gathering Lemmata \ref{lemma:inequalitycase} \ref{lemma:equalitycase}, \ref{lemma:normconvergence} and \ref{lemma:lebesguesemicontinuity}, we obtain the claim of Theorem \ref{thm1} via \eqref{eq:toproveth1}. 

\subsection{Continuity properties of $\L$} \label{sec:contL}
While Lemma \ref{lemma:lebesguesemicontinuity} is targeted to the proof of Theorem \ref{thm1}, it apparently creates a lack of symmetry between the above argument and the standard case of nodal interpolation of functions, where the Lebesgue constant is continuous with respect to perturbations in the Hausdorff distance of the underlying nodes \cite{PiazzonStability,VianelloStability}. Indeed, convergence of interpolation nodes $ \xi_i^{(\varepsilon)} $ in Hausdorff distance translates, in our setting, to point-wise convergence (in the sense of \eqref{eq:pointwiseconvergence}) of the sampling functionals $ \mathcal{T} ^{(\varepsilon)} $, which is a much weaker property than the uniform convergence of $ \mathcal{T} ^{(\varepsilon)} $ introduced in \eqref{eq:uniformconvergence} and used in Lemma \ref{lemma:normconvergence}. To resolve such an asymmetry, we claim that we can in fact prove a stronger version of   \eqref{eq:Llowersemicontinuity} under weaker assumptions.

\begin{proposition}[Continuity of the Lebesgue constant]\label{prop:continuityL}
Let $ \mathcal{T}^{(\varepsilon)} \subset \averaging (E) $ be point-wise convergent to the $\spaceV(E)$-determining subset $\mathcal T$ of $\averaging (E) $ as $ \varepsilon \to 0 $. Then $\mathcal T^{(\varepsilon)}$ is $\spaceV(E)$-determining, for $\varepsilon$ small enough. Moreover one has
\begin{equation}\label{eq:continuityL}
 \lim_{\varepsilon \to 0} \L (\mathcal{T}^{(\varepsilon)}, \spaceV (E) ) = \L (\mathcal{T}, \spaceV (E) ) .
\end{equation}
\end{proposition}
Before proving Proposition \ref{prop:continuityL}, we point out that its combination with Theorem \ref{thm1} directly leads to the following:
\begin{corollary}[Continuity of $\opnorm{P^{(\varepsilon)}}$] \label{cor:continuityoperators}
Let $ \mathcal{T}^{(\varepsilon)} \subset \averaging (E) $ be point-wise convergent to the $\spaceV(E)$-determining subset $\mathcal T$ of $\averaging (E) $ as $ \varepsilon \to 0 $. Assume, for any $\varepsilon>0$, that $\haus^k(\support T_i^{(\varepsilon)}\cap \support T_j^{(\varepsilon)})=0$, whenever $i\neq j$. Similarly assume that $\haus^k(\support T_i\cap \support T_j)=0$ if $i\neq j$. Then
\begin{equation*}
\lim_{\varepsilon\to 0}\opnorm{P^{(\varepsilon)}}=\opnorm{P}\,.
\end{equation*}
\end{corollary}
\begin{proof}[Proof of Proposition \ref{prop:continuityL}]
Let us denote by $\{\eta_h\}_{h=1}^N$ an orthonormal basis of $\spaceV(E)$ with respect to the scalar product $(\cdot,\cdot)_{\mathcal T,w}$ and introcuce the Gram matrices $G_{j,h}=(\eta_j,\eta_h)_{\mathcal T,w}=\delta_{j,h}$, $G_{j,h}^{(\varepsilon)}=(\eta_j,\eta_h)_{\mathcal T^{(\varepsilon)},w}$, where $j,h=1,\dots,N$. We compute
\begin{align*}
&|G_{j,h}^{(\varepsilon)}-G_{j,h}|=\left|\sum_{i=1}^{+\infty}\left(T_i^{(\varepsilon)}(\eta_j)T_i^{(\varepsilon)}(\eta_h)- T_i(\eta_j)T_i(\eta_h)\right)w_i  \right|\\
=& \sum_{i=1}^{+\infty}|T_i(\eta_j)||(T_i^{(\varepsilon)}-T_i)(\eta_h)|w_i+ \sum_{i=1}^{+\infty}|T_i^{(\varepsilon)}(\eta_h)||(T_i^{(\varepsilon)}-T_i)(\eta_j)|w_i\\
\leq& \|\mathcal T(\eta_j)\|_{\ell^\infty}\|\mathcal T_i^{(\varepsilon)}(\eta_h)-\mathcal T(\eta_h)\|_{\ell_w^1} + \|\mathcal T_i^{(\varepsilon)}(\eta_h)\|_{\ell^\infty}\|\mathcal T_i^{(\varepsilon)}(\eta_j)-\mathcal T(\eta_j)\|_{\ell_w^1}\\
\leq& \|\eta_j\|_0\|\mathcal T_i^{(\varepsilon)}(\eta_h)-\mathcal T(\eta_h)\|_{\ell_w^1} + \|\eta_h\|_0\|\mathcal T_i^{(\varepsilon)}(\eta_j)-\mathcal T(\eta_j)\|_{\ell_w^1}.
\end{align*} 
Our assumption on the point-wise convergence of $\mathcal T^{(\varepsilon)}$ implies that $\|\mathcal T_i^{(\varepsilon)}(\eta_h)-\mathcal T(\eta_h)\|_{\ell_w^1}\to 0$ as $\varepsilon\to 0$. Note that we can improve this convergence to uniform in $h\in\{1,\dots,N\}$ (due to the finite dimensionality of $\spaceV(E)$), that is
\begin{equation}\label{eq:duetofiniteN}
\lim_{\varepsilon\to 0} \max_{1\leq h\leq N} \|\mathcal T_i^{(\varepsilon)}(\eta_h)-\mathcal T(\eta_h)\|_{\ell_w^1}=0\,.
\end{equation}
Hence $G^{(\varepsilon)}\to G=\mathbb I$ in any matrix norm. This has the following important consequences:
\begin{enumerate}[i)]
\item there exists $\varepsilon_0>0$ such that for any $\varepsilon\in(0,\varepsilon_0)$ $G^{(\varepsilon)}$ is positive definite and thus $\mathcal T^{(\varepsilon)}$ is $\spaceV(E)$-determininig;
\item for any such $\varepsilon$ there exists a   basis $\{\eta_h^{(\varepsilon)}\}_{h=1}^N$ of $\spaceV(E)$ that is orthonormal with respect to $(\cdot,\cdot)_{\mathcal T^{(\varepsilon)},w}$, and such that
\begin{equation}\label{eq:estimateetaII}
\lim_{\varepsilon\to 0}\max_{h=1,\dots,N}\|\eta_h-\eta_h^{(\varepsilon)}\|_0=0\,;
\end{equation}
\item for any $\varepsilon$ as above, denoting by $K_i:=\sum_{h=1}^N T_i(\eta_h)\eta_h$ the \Riz representer of $T_i$ in the space $(\spaceV(E),(\cdot,\cdot)_{\mathcal T,w})$ and by $K_i^{(\varepsilon)}:=\sum_{h=1}^N T_i^{(\varepsilon)}(\eta_h^{(\varepsilon)})\eta_h^{(\varepsilon)}$ the \Riz representer of $T_i^{(\varepsilon)}$ in the space $(\spaceV(E),(\cdot,\cdot)_{\mathcal T^{(\varepsilon)},w})$, we have
\begin{equation}\label{eq:convergenceofkernel}
\lim_{\varepsilon\to 0} \sum_{i=1}^{+\infty}\|K_i-K_i^{(\varepsilon)}\|_0w_i=0\,.
\end{equation}
\end{enumerate}
While \emph{i}) is immediate and \emph{ii}) can be derived from it (via e.g., the continuity of matrix eigendecomposition),   \eqref{eq:convergenceofkernel} is less evident. In order to show it we first write
\begin{align*}
&K_i-K_i^{(\varepsilon)} =\sum_{h=1}^NT_i(\eta_h^{(\varepsilon)}-\eta_h)\eta_h^{(\varepsilon)}-T_i(\eta_h)(\eta_h^{(\varepsilon)}-\eta_h)+(T_i-T_i^{(\varepsilon)})(\eta_h^{(\varepsilon)})\eta_h^{(\varepsilon)}\,,
\end{align*}
and then derive
\begin{align*}
& \sum_{i=1}^{+\infty}\|K_i-K_i^{(\varepsilon)}\|_0w_i \\ 
\leq & \sum_{i=1}^{+\infty}\sum_{h=1}^N \|\eta_h^{(\varepsilon)}-\eta_h\|_0\|\eta_h^{(\varepsilon)}+\eta_h\|_0w_i + \sum_{i=1}^{+\infty}\sum_{h=1}^N\|\eta_h^{(\varepsilon)}\|_0|(T_i-T_i^{(\varepsilon)})(\eta_h)| w_i \\
\leq & N\max_{h=1,\dots,N}\|\eta_h^{(\varepsilon)}-\eta_h\|_0\|\eta_h^{(\varepsilon)}+\eta_h\|_0 \|w\|_{\ell^1} + \sum_{h=1}^N\|\eta_h^{(\varepsilon)}\|_0 \max_j\|(\mathcal T-\mathcal T^{(\varepsilon)})(\eta_j)\|_{\ell_w^1}\,.
\end{align*}
Combining this estimate with   \eqref{eq:estimateetaII} and   \eqref{eq:duetofiniteN} leads to \eqref{eq:convergenceofkernel}.

We are left to prove that   \eqref{eq:convergenceofkernel} implies   \eqref{eq:continuityL}. To this aim we exploit   \eqref{eq:rephrasingL} and we write
\begin{align*}
\L(\mathcal T^{(\varepsilon)},\spaceV (E) )=&\sup_{T\in \averaging(E)}\sum_{i=1}^{+\infty}|T(K_i^{(\varepsilon)})|w_i\\
\leq& \sum_{i=1}^{+\infty}\sup_{T\in \averaging(E)}|T(K_i^{(\varepsilon)}-K_i)|w_i+ \sup_{T\in \averaging(E)}\sum_{i=1}^{+\infty}|T(K_i)|w_i\\
=&\sum_{i=1}^{+\infty}\|K_i^{(\varepsilon)}-K_i\|_0w_i+ \L(\mathcal T,\spaceV (E))\,,
\end{align*}
and similarly $\L(\mathcal T,\spaceV (E)) \leq \sum_{i=1}^{+\infty}\|K_i-K_i^{(\varepsilon)}\|_0w_i+ \L(\mathcal T^{(\varepsilon)},\spaceV (E))$. Hence
$$|\L(\mathcal T^{(\varepsilon)},\spaceV (E))-\L(\mathcal T,\spaceV (E))|\leq \sum_{i=1}^{+\infty}\|K_i-K_i^{(\varepsilon)}\|_0w_i$$
and the right hand side tends to $0$ as $\varepsilon\to 0$ due to \eqref{eq:convergenceofkernel}.
\end{proof}

\revised{
\subsection{Interplay between Hilbert and Banach structures} \label{sec:BanachHilbert}

Throughout Section \ref{sec:Background}, we have encountered different norms for $ \testforms (E) $. In particular, we observed that the norm $ \Vert \cdot \Vert_0 $ only makes $ \testforms (E) $ a Banach space. This makes the analysis of the relationship between $ \Vert I - P \Vert_{\mathrm{op}} $ and $ \Vert P \Vert_{\mathrm{op}} $ more involved.

Recall that, for every non-trivial projector $ P:X\rightarrow X $ on a normed space $X$, there exists $ 1 \leq C \leq 2 $ such that
\begin{equation}
    \Vert I - P \Vert_{\mathrm{op}} \leq C \Vert P \Vert_{\mathrm{op}} .
\end{equation}
When $X$ is in particular a Hilbert space, the constant $C$ can be taken equal to $1$.
More generally, the sharpest constant $C$ can be characterized \cite[Theorem 3]{Stern} as
$$ C = \min \{ C_{BM}, 1 + \Vert P \Vert_{\mathrm{op}}^{-1} \} ,$$
where $ C_{BM} $ is the Banach-Mazur constant. This quantity, roughly speaking, measures how much a normed space fails to be a inner product space. Despite $ C_{BM} $ may be difficult to compute, in the interpolation case Theorem \ref{thm1} allows to claim something about $ C $.

\begin{lemma} \label{lem:lowerboundL}
    Let $ \mathcal{T} := \{ T_1, \ldots, T_N \} $ be a unisolvent collection of currents for $ \spaceV (E) $, with Lebesgue constant $ \L(\mathcal T, \spaceV (E)) $. Let $\Pi: \testforms(E)\rightarrow \spaceV (E)\subset \testforms(E) $ be the corresponding interpolation operator. One has
    $$ \L(\mathcal T, \spaceV (E)) \geq 1 .$$
    In particular, if $ \haus^k(\mathrm{supp}(T_i) \cap \mathrm{supp}( T_j)) = 0 $ for $ i \ne j $, one has
    $$ \Vert \Pi \Vert_{\mathrm{op}} \geq 1. $$
\end{lemma}

\begin{proof}
    Since $ \mathrm{supp}(T_j) $ is a compact set for each $ j = 1, \ldots, N $, one has
    $$ \L(\mathcal T,\spaceV (E)) =\sup_{T\in \averaging(E)} \sum_{i=1}^N \left|  T(\omega_i)\right| \geq \sum_{i=1}^N \left|  T_j(\omega_i)\right| = 1 , $$
    due to the duality relationship of the Lagrange functions. The second part of the claim follows directly from Theorem \ref{thm1}.
\end{proof}

As an immediate consequence of Lemma \ref{lem:lowerboundL}, one again retrieves that, under the hypotheses of such a result, $ C \leq 2 $. On the other hand, in many situations one deals with rather large Lebesgue constants, making $ C \leq 1 + \Vert \Pi \Vert_{\mathrm{op}}^{-1} = 1 + \L(\mathcal T,\spaceV (E))^{-1} $ close to $ 1 $. How far $ \L(\mathcal T,\spaceV (E)) $ is from $ 1 $ generally depends both on $ \mathcal T $ and $ \spaceV (E) $. 

For polynomial spaces, a lower bound for the operator norm of any projector $ P $ can be given. In the univariate case, this is a consequence of the Kharshiladze-Lozinski theorem \cite[\S 3, Theorem 1]{CheneyLight}. In the multivariate case this requires some further hypotheses on the domain; results with explicit constants are available for the unit $n$-ball \cite{Sundermann}. In both cases, the resulting lower bound is a function greater than or equal to $ 1 $, which increases with the polynomial degree. Working component-wise, one may interpret these bounds as (rough) lower bounds also for polynomial differential forms.
}

\section{Dependence of $\L$ on the reference domain: proof of Theorem \ref{thm2}} \label{sect:changing}
Most results regarding Lagrange interpolation of differential forms, usually applied to polynomial differential forms, are obtained on a reference set $ \widehat{E} $, which is usually taken to be the standard $n$-simplex
, and then mapped to a set $ E $ in the physical domain. While this procedure does not affect qualitative features, such as unisolvence, 
for $ k = 1, \ldots, n-1 $ it can substantially change the quantitative features of the approximation scheme, such as the Lebesgue constant $ \L(\mathcal{T}, \spaceV (E) ) $, as noted in \cite[\S 6]{ABRCalcolo}. This issue essentially arises from the lack of invariance of the averaging currents under the pushforward operator, discussed in Remark \ref{rmk:pullbackaveraged}.  \revised{The principal aim of this section thus consists in proving the second main result of the paper:

\noindent \textbf{Theorem $\boldsymbol{2}$} (Dependence of $ \L $ on $ E $)\textbf{.}
\emph{Let $ \varphi:\widehat E\rightarrow E $ be a $\mathscr C^1$-diffeomorphism. Then one has
	\begin{equation*}
		\L ({\mathcal T}, \mathscr V^k (E) ) \leq \left\|\prod_{j=1}^k \sigma^{(\varphi)}_j\right\|_{\mathscr C^0(\widehat E)} \left\| \frac 1{\prod_{j=n-k+1}^n\sigma^{(\varphi)}_j}\right\|_{\mathscr C^0(\widehat E)} \L (\widehat{\mathcal T}, \mathscr V^k (\widehat{E}) )\,,
	\end{equation*}
where $\left\{\sigma^{(\varphi)}_1,\dots,\sigma^{(\varphi)}_n\right\}$ are the singular value functions of the differential of $\varphi$ arranged in a non-increasing order, and $\|\cdot\|_{\mathscr C^0(\widehat E)}$ denotes the uniform norm of continuous functions on $\widehat E$.}}

Following Remark \ref{rmk:pullbackaveraged}, before proving Theorem \ref{thm2}, we need to study how the $\haus^k$-measure of $k$-dimensional subsets of $\widehat E$ varies under smooth mappings of the $n$-dimensional set $\widehat E$. We start from a closer look to the affine case:

\begin{example} \label{ex:volumeksimplex}
    Let $ \widehat S \subset \widehat E \subset \R^n $ be a $k$-simplex spanned by vertices $ \{\widehat v_0, \ldots, \widehat v_k \} $. 
    Denote by $ \widehat V $ the matrix whose columns are the vectors $ \{\widehat v_0, \ldots, \widehat v_k \} $. Then  $ \mathcal{H}^k (\widehat S) $, which is in fact the $k$-volume of $\widehat S  $, is 
    $$ \mathcal{H}^k (\widehat S) = \frac{1}{k!} \sqrt{\det (\widehat V^\top \widehat V)}.$$ 
    Let $ \varphi : \widehat{E} \to E:= \varphi(\widehat{E}) $ be an affine mapping with invertible linear part $ A \in \mathrm{GL}(n, \R) $ and $ S = \varphi (\widehat S) $. 
Denoting by $ V $ the matrix whose columns are the vectors $ \{\varphi(\widehat v_0), \ldots, \varphi(\widehat v_k) \} $, one has 
\begin{equation*}
    \mathcal{H}^k (S) = \frac{1}{k!} \sqrt{\det (V^\top V)} = \frac{1}{k!} \sqrt{\det ((A \widehat V)^\top A \widehat V)} = \frac{1}{k!} \sqrt{\det (\widehat V^\top A^\top A \widehat V)} .
\end{equation*}
Denoting by $ \sigma_1 \geq \ldots \geq \sigma_n $ the singular values of the matrix $ A $ and performing an SVD decomposition (see \cite[Theorem 2.4.1]{Golub} or \cite[p. 854]{Strang}) as $ A = U \Sigma W^T $, one finds that 
\begin{equation*}
   \frac{1}{k!} \sqrt{\det (\widehat V^\top A^\top A \widehat V)} = \frac{1}{k!} \sqrt{\det (\widehat V^\top W \Sigma^\top U^\top U \Sigma W^T \widehat V)} = \frac{1}{k!} \sqrt{\det (\widehat V^\top W \Sigma^\top \Sigma W^T \widehat V)} .
\end{equation*}
Exploiting the interlacing result \cite[Theorem 7.3.9]{HoJo85}, the latter term can be bounded from below as 
\begin{equation*}
   \sqrt{\det (\widehat V^\top W \Sigma^\top \Sigma W^T \widehat V)} \geq \left( \prod_{i=n-k+1}^n \sigma_{i} \right) \sqrt{\det (\widehat V^\top W W^T \widehat V)} = \left( \prod_{i=n-k+1}^n \sigma_{i} \right) \sqrt{\det (\widehat V^\top  \widehat V)} 
\end{equation*}
and, likewise, from above as
\begin{equation*}
   \sqrt{\det (\widehat V^\top W \Sigma^\top \Sigma W^T \widehat V)} \leq \left( \prod_{i=1}^{k} \sigma_{i} \right) \sqrt{\det (\widehat V^\top W W^T \widehat V)} = \left( \prod_{i=1}^{k} \sigma_{i} \right) \sqrt{\det (\widehat V^\top \widehat V)}.
\end{equation*}
Plugging this into the computation of $ \haus^k (S) = 1/k! \sqrt{ \det V^\top V} $, one finds 
\begin{equation*}
    \left( \prod_{i=n-k+1}^n \sigma_{i} \right) \mathcal{H}^k (\widehat S) \leq \mathcal{H}^k (S) \leq \left( \prod_{i=1}^{k} \sigma_{i} \right) \mathcal{H}^k (\widehat S).
\end{equation*}
Notice that equality holds when $ k = 0$ or $ k = n $.
\end{example}

Via the Area Formula (see, e.g., \cite[\S 3.3]{EvansGariepy}), Example \ref{ex:volumeksimplex} can be extended to $k$-rectifiable sets and $ \mathscr C^1 $-diffeomorphisms:

\begin{lemma}\label{lem:fede}
Let $ K \subset \R^n$ be a $k$-rectifiable set such that  $\haus^k (K) < + \infty$. Let $U$ be a neighborhood of $ K $ in $\R^n$ and let $\varphi:U\rightarrow \R^n$ be a $\mathscr C^1$-diffeomorphism on its image. Then one has
\begin{equation}\label{eq:mappingareas}
\left(\left\|\prod_{j=n-k+1}^n (\sigma^{(\varphi)}_j)^{-1}\right\|_{\mathscr{C}^0 (\bar K)} \right)^{-1} \haus^k(K)\leq \haus^k(\varphi(K))\leq \left\|\prod_{j=1}^k \sigma^{(\varphi)}_j\right\|_{\mathscr{C}^0 (\bar K)} \haus^k(K)\,,
\end{equation} 
where $\sigma_j^{(\varphi)}:U\rightarrow (0,+\infty)$ are the singular values functions of the matrix $D\varphi(\cdot)$ arranged in non-increasing order, i.e., for any $x\in U$, one has $\sigma_l^{(\varphi)}(x)\geq \sigma_{j}^{(\varphi)}(x)$ for $l>j$.
\end{lemma}			
\begin{proof}
Recall that $K=\cup_{j=0}^\infty K_j$, where $K_{0}$ has zero $\haus^k$-measure, $K_i\cap K_j=\emptyset$ for distinct indices, and, for any $i>0$,  $K_i$ is a measurable subset of a $\mathscr C^1$ embedded $k$-submanifold $S_i$. For each $S_i$ consider an atlas $\{(\phi_\alpha^i:U_\alpha^i\rightarrow V_\alpha^i)\}$, $V_\alpha^i\subset \R^k$, and a partition of unity $\{\rho_\alpha^i\}$ subordinated to the atlas, i.e., $\rho_\alpha^i\in \mathscr C^\infty_c(U_\alpha^i,[0,1])$,  $\sum_\alpha \rho_\alpha^i=1$ on $S_i$. Finally, let $\psi_\alpha^i:=(\phi_\alpha^i)^{-1}:V_\alpha^i\rightarrow U_\alpha^i$.

Since $\varphi$ is a $\mathscr C^1$-diffeomorphism, it induces an atlas of $\varphi(S_i)$ and the partition of unity $\tilde \rho_\alpha^i:=\rho_\alpha^i\circ\varphi^{-1}$. We can compute $\haus^k(\varphi(K_i))$ by the Area Formula:
\begin{align*}
&\haus^k(\varphi(K_i))\\
=& \int_{\varphi(K_i)} \de \haus^k=\sum_\alpha \int_{\varphi(U_\alpha^i\cap K_i)} \tilde \rho_\alpha^i(\tilde x) \de \haus^k(\tilde x)\\
=&\sum_\alpha \int_{V_\alpha^i\cap \phi_\alpha^i(K_i)} \rho_\alpha^i(\psi_\alpha^i(y)) \Jac{D (\varphi\circ\psi_\alpha^i)}(y)\de \haus^k(y)\\
=&\sum_\alpha \int_{V_\alpha^i\cap \phi_\alpha^i(K_i)} \rho_\alpha^i(\psi_\alpha^i(y)) \Jac{D\varphi(\psi_\alpha^i(y))D\psi_\alpha^i(y)}\de \haus^k(y)\\
=&\sum_\alpha \int_{V_\alpha^i\cap \phi_\alpha^i(K_i)} \rho_\alpha^i(\psi_\alpha^i(y)) 
\sqrt{\det\left[(D\psi_\alpha^i(y))^\top (D\varphi(\psi_\alpha^i(y)))^\top   D\varphi(\psi_\alpha^i(y))D\psi_\alpha^i(y) \right]} \de \haus^k(y)\,.
\end{align*}
Here we denoted by $\Jac{A}$ the Jacobian of the $n$ by $k$ matrix $A$, see, e.g., \cite[\S 3.2]{EvansGariepy}.

Consider the singular value decomposition $D \psi_\alpha^i(y)=U(y)\Sigma(y)V(y)^\top$. Let 
$$\Sigma(y)^\top=[\Sigma_0^\top(y)|\mathbb O_{k,n-k}]:=[\mathrm{diag}(\sigma_1(y),\dots\sigma_k(y))|\mathbb O_{k,n-k}]\,,$$
where $\mathbb O_{k,n-k}$ is the zero matrix of size $k\times (n-k)$, and $\mathrm{diag}(\sigma_1(y),\dots\sigma_k(y))$ is the $k\times k$ diagonal matrix with $\sigma_1(y),\dots\sigma_k(y)$ as diagonal entries. Then we compute
\begin{align*}
&\sqrt{\det\left[(D\psi_\alpha^i(y))^\top (D\varphi(\psi_\alpha^i(y)))^\top   D\varphi(\psi_\alpha^i(y))D\psi_\alpha^i(y) \right]}\\
=&\sqrt{\det(V(y)\Sigma^\top(y)U^\top(y)(D\varphi(\psi_\alpha^i(y)))^\top   D\varphi(\psi_\alpha^i(y))U(y)\Sigma(y)V(y)^\top)}\\
=&\sqrt{\det(\Sigma_0^\top(y)\mathbb I_{n,k}^\top A_{\alpha,i}^\top(y)A_{\alpha,i}(y)\mathbb I_{n,k}\Sigma_0(y))}\\
=&\left(\prod_{j=1}^k\sigma_j(y)\right) \; \sqrt{\det(\mathbb I_{n,k}^\top A_{\alpha,i}^\top(y)A_{\alpha,i}(y)\mathbb I_{n,k})}\,,
\end{align*}
where $A_{\alpha,i}(y):=D\varphi(\psi_\alpha^i(y))U(y)$, and $(\mathbb I_{n,k})_{l,m}:=\delta_{l,m}.$  Using the Cauchy-Binet formula \cite[p. 68]{Shafarevich}, one has
\begin{align}\label{eq:mainareaestimate}
&\haus^k(\varphi(K_i))=\sum_\alpha \int_{V_\alpha^i\cap \phi_\alpha^i(K_i)} \rho_\alpha^i(\psi_\alpha^i(y))\left(\prod_{j=1}^k\sigma_j(y)\right)\; \sqrt{\det(\mathbb I_{n,k}^\top A_{\alpha,i}^\top(y)A_{\alpha,i}(y)\mathbb I_{n,k})} \de \haus^k(y)\notag \\
\leq& \sup_{\alpha}\left\| \sqrt{\det(\mathbb I_{n,k}^\top A_{\alpha,i}^\top(y)A_{\alpha,i}(y)\mathbb I_{n,k})}  \right\|_{L^{\infty}(V_\alpha^i\cap \phi_\alpha^i(K_i))}\cdot \sum_\alpha \int_{V_\alpha^i\cap \phi_\alpha^i(K_i)} \rho_\alpha^i(\psi_\alpha^i(y))\left(\prod_{j=1}^k\sigma_j(y)\right) \de \haus^k(y)\notag \\
=& \sup_{\alpha}\left\| \sqrt{\det(\mathbb I_{n,k}^\top A_{\alpha,i}^\top (y)A_{\alpha,i}(y)\mathbb I_{n,k})}  \right\|_{L^{\infty}(V_\alpha^i\cap \phi_\alpha^i(K_i))}\cdot \sum_\alpha \int_{V_\alpha^i\cap \phi_\alpha^i(K_i)} \rho_\alpha^i(\psi_\alpha^i(y))\Jac{D \psi_\alpha^i(y)} \de \haus^k(y)\notag \\
=& \sup_{\alpha}\left\| \sqrt{\det(\mathbb I_{n,k}^\top A_{\alpha,i}^\top (y)A_{\alpha,i}(y)\mathbb I_{n,k})}  \right\|_{L^{\infty}(V_\alpha^i\cap \phi_\alpha^i(K_i))}\cdot \sum_\alpha \int_{U_\alpha^i\cap K_i} \rho_\alpha^i(x) \de \haus^k(x)\notag \\
=& \sup_{\alpha}\left\| \sqrt{\det(\mathbb I_{n,k}^\top A_{\alpha,i}^\top (y)A_{\alpha,i}(y)\mathbb I_{n,k})}  \right\|_{L^{\infty}(V_\alpha^i\cap \phi_\alpha^i(K_i))}\cdot \haus^k(K_i)
\end{align}
If $A$ is an $l \times m$ matrix with $l\geq m$ having singular values $\sigma_1 \geq \dots \geq \sigma_m$, then the singular values $\sigma_1^{(1)} \geq  \dots \geq \sigma_{m-1}^{(1)}$ of the matrix $A^{(1)}$ obtained removing a column from $A$ satisfy
\begin{equation*}\label{eq:estimatesigma}
\sigma_1\geq \sigma_1^{(1)}\geq \sigma_2\geq \sigma_2^{(1)}\geq \cdots\geq\sigma_{m-1}\geq\sigma_{m-1}^{(1)}\geq \sigma_m\,,
\end{equation*}
see, e.g., \cite[Theorem 7.3.9]{HoJo85}. Taking $l,m=n$, and iterating $n-k$ times the interlacing inequalities of \eqref{eq:estimatesigma} (removing the last column at each step), we can estimate each of the sigular values 
$ \sigma_j^{(n-k)}$ of $A^{(n-k)}:=A\mathbb I_{n,k}$ by the corresponding singular value of $A$, i.e.,
\begin{equation*}
\sigma_j^{(n-k)}\leq \sigma_j,\;\;j=1,2,\dots,k\,.
\end{equation*}
Applying this result to $A_{\alpha,i}(y)$ for any $y\in V_{\alpha}^i$, the $j$-th singular value of $A_{\alpha,i}^{(n-k)}$ is bounded from above by the $j$-th singular value of $A_{\alpha,i}(y):=D\varphi(\psi_\alpha^i(y))U(y)$, which equals the $j$-th singular value $\sigma_j^{(\varphi)}(\psi_\alpha^i(y))$ of $D\varphi(\psi_\alpha^i(y))$ since $ U(y) $ is orthogonal. 
Thus
\begin{equation}\label{eq:productestimate}
\max_{y\in V_\alpha^i\cap \phi_\alpha^i(K_i)}\sqrt{\det(\mathbb I_{n,k}^\top A_{\alpha,i}^\top (y)A_{\alpha,i}(y)\mathbb I_{n,k})}\leq \max_{x\in U_\alpha^i\cap K_i}\prod_{j=1}^k \sigma_j^{(\varphi)}(x), \quad\forall i, \; \forall \alpha\,.
\end{equation}
The combination of \eqref{eq:mainareaestimate} with \eqref{eq:productestimate} proves 
$$ \haus^k(\varphi(K_i))\leq \sup_\alpha \left(\max_{x\in U_\alpha^i\cap K_i}\prod_{j=1}^k \sigma_j^{(\varphi)}(x)\right) \haus^k(K_i) = \sup_{x \in K_i} \left( \prod_{j=1}^k \sigma_j^{(\varphi)}(x) \right) \haus^k(K_i) ,$$
whence, summing over $i$, the rightmost inequality in \eqref{eq:mappingareas} follows. The leftmost one can be obtained repeating the same reasoning, replacing $\varphi$ by $\varphi^{-1}$ and using the Inverse Function Theorem \cite[Theorem 1.1.7]{Hormander}. \myqed
\end{proof}

\begin{remark}
While not needed for our purposes, one may wonder if Lemma \ref{lem:fede} may be extended to weaker assumptions using finer tools in geometric measure theory such as decomposability bundles \cite{AlMa16,BoNiRi24}. 
On the one hand, it is possible to expect that metric properties behave regularly under bi-Lipschitz mappings, but on the other one the differentials of $\varphi$ and $\varphi^{-1}$ might not be well-defined on a $k$-rectifiable set, which is negligible with respect to $ n$-dimensional measure. 
\end{remark}

We are now ready to prove Theorem \ref{thm2}:
\begin{proof}[Proof of Theorem \ref{thm2}]
Let us denote by $\widehat S_i:=\support \widehat T_i$ and by $ \widehat \tau_i $ its orientation. Then, using the notation $S_i=\varphi(\widehat S_i)$ and the set of elements of $\averaging(E)$ introduced in   \eqref{newproduct}, we can write
$$T_i:=\frac{[S_i, \tau_i]}{\haus^k(S_i)}=\frac{\haus^k(\widehat S_i)}{\haus^k( S_i)}\varphi_* \widehat T_i\,.$$
From \eqref{neworthogonalbasis} recall that $\{\widehat \eta_h\}_{h=1}^N$ is an orthonormal basis of $\spaceV(\widehat E)$ and, setting $\psi=\varphi^{-1}$, the set $\{\eta_h\}_{h=1}^N:=\{\psi^*\widehat \eta_h\}_{h=1}^N$ is an orthonormal basis of $\spaceV(E)$ with respect to the product $(\cdot,\cdot)_{\mathcal T,w}$, with $w_i=\left(\frac{\haus^k( S_i)} {\haus^k(\widehat S_i)}\right)^2\widehat w_i$. Let $T\in \averaging(E)$ and denote by $S$ its support with orientation $ \tau $, and recall that $ \widehat S $ and $ \widehat \tau $ are related to $ S $ and $ \tau $ as in Remark \ref{rmk:pullbackaveraged}. We compute:
\begin{align*}
&\sum_{i=1}^M \left\vert \sum_{h=1}^N w_i T_i (\eta_h) T (\eta_h) \right \vert
= \sum_{i=1}^M \left\vert \sum_{h=1}^N \left(\frac{\haus^k( S_i)} {\haus^k(\widehat S_i)}\right)^2\widehat w_i \frac{\haus^k(\widehat S_i)}{\haus^k( S_i)}\varphi_* \widehat T_i (\eta_h) T (\eta_h) \right \vert\\
=& \sum_{i=1}^M \left\vert \sum_{h=1}^N \frac{\haus^k( S_i)} {\haus^k(\widehat S_i)}\widehat w_i  \widehat T_i (\varphi^*\psi^*\widehat\eta_h) T (\psi^*\widehat\eta_h) \right \vert
\leq \sup_{i}\frac{\haus^k( S_i)} {\haus^k(\widehat S_i)} \sum_{i=1}^M \left\vert \sum_{h=1}^N \widehat w_i  \widehat T_i (\widehat\eta_h) \psi_*T (\widehat\eta_h) \right \vert\\
=& \sup_{i}\frac{\haus^k( S_i)} {\haus^k(\widehat S_i)} \frac{\haus^k(\psi(S))}{\haus^k(S)} \sum_{i=1}^M \left\vert \sum_{h=1}^N \widehat w_i \widehat T_i (\widehat\eta_h) \frac{[\psi(S), \widehat \tau]}{\haus^k(\psi(S))} (\widehat\eta_h) \right \vert\\
=& \sup_{i}\frac{\haus^k( S_i)} {\haus^k(\widehat S_i)} \frac{\haus^k(\widehat S)}{\haus^k(S)}  \sum_{i=1}^M \left\vert \sum_{h=1}^N \widehat w_i  \widehat T_i (\widehat\eta_h) \frac{[\widehat S, \widehat \tau]}{\haus^k(\widehat S)} (\widehat\eta_h) \right \vert\,.
\end{align*}
Hence
\begin{equation}\label{eq:stima1}
\sum_{i=1}^M \left\vert \sum_{h=1}^N w_i T_i (\eta_h) T (\eta_h) \right \vert\leq\sup_{i}\frac{\haus^k( S_i)} {\haus^k(\widehat S_i)}\; \frac{\haus^k(\widehat S)}{\haus^k(S)}\; \L(\widehat T,\spaceV(\widehat E))\,,
\end{equation}
since $\frac{[\widehat S, \widehat \tau]}{\haus^k(\widehat S)}\in \averaging(\widehat E)$. Note that \revised{the supremum over $ i $ is finite, and } we can give an upper bound to the first two factors in the above product using Lemma \ref{lem:fede}:
\begin{equation}\label{eq:stima2}
 \frac{\haus^k(\widehat S)}{\haus^k(S)}\leq \left\| \frac 1{\prod_{j=n-k+1}^n\sigma^{(\varphi)}_j}\right\|_{\mathscr C^0(\widehat E)} \quad \text{and} \quad \sup_{i}\frac{\haus^k( S_i)} {\haus^k(\widehat S_i)}\leq \left\|\prod_{j=1}^k \sigma^{(\varphi)}_j\right\|_{\mathscr C^0(\widehat E)}\,.
\end{equation}
Combining   \eqref{eq:stima1} and   \eqref{eq:stima2} 
gives
\begin{align*}
\L(T,\spaceV( E))=
\left\|\prod_{j=1}^k \sigma^{(\varphi)}_j\right\|_{\mathscr C^0(\widehat E)} \left\|\frac 1{\prod_{j=n-k+1}^n \sigma^{(\varphi)}_j}\right\|_{\mathscr C^0(\widehat E)} \L(\widehat T,\spaceV(\widehat E))\,,
\end{align*}
which concludes the proof. \myqed
\end{proof}

When $ \varphi $ is an affine mapping one gets a strong simplification of Lemma \ref{lem:fede}. Since the differential of $ \varphi $ is the constant matrix $ D \varphi = A $, in such a framework the estimate \eqref{eq:mappingareas} reads:
%
	\begin{equation*} \label{eq:minorazionevolume}
		  \left( \prod_{j=n-k+1}^n \sigma_j \right) \haus^k (S) \leq \haus^k (\varphi(S)) \leq \left( \prod_{j=1}^k \sigma_j \right) \haus^k (S) .
	\end{equation*}
Consequently, under this assumption Theorem \ref{thm2} may be stated as follows.
\begin{corollary} \label{cor:changingsimplex}
Let $ \varphi: \widehat E \to E $ be an affine mapping with linear part $ A $. Let $ \sigma_1 \geq \ldots \geq \sigma_n $ be the singular values of $ A $. 
Then
	\begin{equation} \label{eq:estimationchangingsimplex}
		\L ({\mathcal T}, \spaceV (E) ) \leq \frac{\prod_{j=1}^k \sigma_j}{\prod_{j=n-k+1}^n \sigma_j} \L (\widehat{\mathcal T}, \spaceV (\widehat{E}) ) .
	\end{equation}
\end{corollary}

Corollary \ref{cor:changingsimplex} relates with the literature, where similar results have been proved for $ k = 1 $ (see \cite[Lemma 2 and Proposition 1]{ABRCalcolo}), 
and for a general $ k $ (see \cite[Proposition 3.21]{BruniThesis}) under stricter assumptions. 
These results exploit different techniques, which for $ k > 1 $ in turn yield the weaker result
\begin{equation} \label{eq:oldestimationchangingsimplex}
	\L ({\mathcal T}, \spaceV (E) ) \leq (\mathrm{Cond}_2 (A))^k \, \L (\widehat{\mathcal T}, \spaceV (\widehat{E}) ),
\end{equation}
the symbol $ \mathrm{Cond}_2 $ denoting the \revised{condition number\ } $ \mathrm{Cond}_2 := \sigma_1 / \sigma_n $ of the matrix $ A $ in the $2$-norm. Indeed,  \eqref{eq:estimationchangingsimplex} implies   \eqref{eq:oldestimationchangingsimplex}, since
\begin{align*} \L ({\mathcal T}, \spaceV (E) ) & \leq \frac{\prod_{j=1}^k \sigma_j}{\prod_{j=n-k+1}^n \sigma_j} \L (\widehat{\mathcal T}, \spaceV (\widehat{E}) ) \leq \frac{\sigma_1^k}{\sigma_n^k} \L (\widehat{\mathcal T}, \spaceV (\widehat{E}) ) 
\\ & = (\mathrm{Cond}_2 (A))^k \L (\widehat{\mathcal T}, \spaceV (\widehat{E}) ) .
\end{align*}
Further, \eqref{eq:estimationchangingsimplex} shows that   \eqref{eq:oldestimationchangingsimplex} is pessimistic for $ k > \lfloor \frac{n}{2} \rfloor $. In fact, assume that $ k > \lfloor \frac{n}{2} \rfloor $. Then $ \sigma_k $ appears both at the numerator and the denominator of the right hand side of   \eqref{eq:estimationchangingsimplex}, and hence cancels out. In turn, we immediately deduce that   \eqref{eq:oldestimationchangingsimplex} is improved by
$$ \L ({\mathcal T}, \spaceV (E) ) \leq \min \left\{ (\mathrm{Cond}_2 (A))^k, (\mathrm{Cond}_2 (A))^{n-k} \right\} \L (\widehat{\mathcal T}, \spaceV (\widehat{E}) ) .$$
Notice also that
$$ \min \left\{ (\mathrm{Cond}_2 (A))^k, (\mathrm{Cond}_2 (A))^{n-k} \right\} =
\begin{cases}
	(\mathrm{Cond}_2 (A))^k \quad & \text{if } k \leq \lfloor \frac{n}{2} \rfloor , \\
	(\mathrm{Cond}_2 (A))^{n-k} \quad & \text{if } k > \lfloor \frac{n}{2} \rfloor .
\end{cases} $$

\revised{\begin{remark} One may apply Theorem \ref{thm2} to $\varphi$ and $\varphi^{-1}$ to show that the Lebesgue constant \eqref{eq:defconstantM} does not depend on the supporting domain $ E $ for $ k = 0 $ or $ k = n $. Such a claim is known in the nodal interpolatory framework (see, e.g., \cite{Hesthaven}), and has been explicitly shown for top dimensional forms only for $ n = 1 $, see \cite{BruniErb23}. For $0<k<n$ the dependence of $\L$ on the reference domain $E$ has been only numerically observed, see \cite[\S 6]{ABRCalcolo} and \cite[\S 4.2.1]{BruniThesis}.
\end{remark}}

  \bibliographystyle{abbrv} 
  \bibliography{bibliography}

@book {Lee,
	AUTHOR = {Lee, J. M.},
	TITLE = {Introduction to smooth manifolds},
	SERIES = {Graduate Texts in Mathematics},
	VOLUME = {218},
	EDITION = {Second},
	PUBLISHER = {Springer, New York},
	YEAR = {2013},
	PAGES = {xvi+708},
}

@book {Flanders,
	AUTHOR = {Flanders, H.},
	TITLE = {Differential forms with applications to the physical sciences},
	SERIES = {Dover Books on Advanced Mathematics},
	EDITION = {Second},
	PUBLISHER = {Dover Publications, Inc., New York},
	YEAR = {1989},
	PAGES = {xvi+205},
}

@book {BottTu,
	AUTHOR = {Bott, R. and Tu, L. W.},
	TITLE = {Differential forms in algebraic topology},
	SERIES = {Graduate Texts in Mathematics},
	VOLUME = {82},
	PUBLISHER = {Springer-Verlag, New York-Berlin},
	YEAR = {1982},
	PAGES = {xiv+331},
}

@book {Bossavit,
	AUTHOR = {Bossavit, A.},
	TITLE = {Computational electromagnetism},
	SERIES = {Electromagnetism},
	NOTE = {Variational formulations, complementarity, edge elements},
	PUBLISHER = {Academic Press, Inc., San Diego, CA},
	YEAR = {1998},
	PAGES = {xx+352},
}

@book {Whitney,
	AUTHOR = {Whitney, H.},
	TITLE = {Geometric integration theory},
	PUBLISHER = {Princeton University Press, Princeton, NJ},
	YEAR = {1957},
	PAGES = {xv+387},
}

@book {ErnI,
	AUTHOR = {Ern, A. and Guermond, J.-L.},
	TITLE = {Finite elements {I}---{A}pproximation and interpolation},
	SERIES = {Texts in Applied Mathematics},
	VOLUME = {72},
	PUBLISHER = {Springer, Cham},
	YEAR = {2021},
	PAGES = {xii+325},
}

@book {Golub,
	AUTHOR = {Golub, G. H. and Van Loan, C. F.},
	TITLE = {Matrix computations},
	SERIES = {Johns Hopkins Studies in the Mathematical Sciences},
	EDITION = {Fourth},
	PUBLISHER = {Johns Hopkins University Press, Baltimore, MD},
	YEAR = {2013},
	PAGES = {xiv+756},
}

@book {Bhatia,
    AUTHOR = {Bhatia, R.},
     TITLE = {Matrix analysis},
    SERIES = {Graduate Texts in Mathematics},
    VOLUME = {169},
 PUBLISHER = {Springer-Verlag, New York},
      YEAR = {1997},
     PAGES = {xii+347},
}

@article {Harrison,
	AUTHOR = {Harrison, J.},
	TITLE = {Continuity of the integral as a function of the domain},
	NOTE = {Dedicated to the memory of Fred Almgren},
	JOURNAL = {J. Geom. Anal.},
	FJOURNAL = {The Journal of Geometric Analysis},
	VOLUME = {8},
	YEAR = {1998},
	NUMBER = {5},
	PAGES = {769--795},
}

@article {AlonsoRapettiLebesgue,
	AUTHOR = {Alonso Rodr\'{\i}guez, A. and Rapetti, F.},
	TITLE = {On a generalization of the {L}ebesgue's constant},
	JOURNAL = {J. Comput. Phys.},
	FJOURNAL = {Journal of Computational Physics},
	VOLUME = {428},
	YEAR = {2021},
	PAGES = {Paper No. 109964, 4},
}

@article {BruniErb23,
	title={Polynomial Interpolation of Function Averages on Interval Segments}, 
	author={Bruni Bruno, L. and Erb, W.},
	year={2024},
	JOURNAL = {SIAM J. Numer. Anal.},
	FJOURNAL = {SIAM Journal on Numerical Analysis},
	volume = {92},
    number = {4},
    pages = {1759-1781},
}

@article {Nedelec1,
	AUTHOR = {N\'{e}d\'{e}lec, J.-C.},
	TITLE = {Mixed finite elements in {${\bf R}\sp{3}$}},
	JOURNAL = {Numer. Math.},
	FJOURNAL = {Numerische Mathematik},
	VOLUME = {35},
	YEAR = {1980},
	NUMBER = {3},
	PAGES = {315--341},
}

@article {Nedelec2,
	AUTHOR = {N\'{e}d\'{e}lec, J.-C.},
	TITLE = {A new family of mixed finite elements in {${\bf R}^3$}},
	JOURNAL = {Numer. Math.},
	FJOURNAL = {Numerische Mathematik},
	VOLUME = {50},
	YEAR = {1986},
	NUMBER = {1},
	PAGES = {57--81},
}

@article {ArnoldFalkWintherActa,
	AUTHOR = {Arnold, D. N. and Falk, R. S. and Winther, R.},
	TITLE = {Finite element exterior calculus, homological techniques, and
	applications},
	JOURNAL = {Acta Numer.},
	FJOURNAL = {Acta Numerica},
	VOLUME = {15},
	YEAR = {2006},
	PAGES = {1--155},
}

@article {Gopalakrishnan,
	AUTHOR = {Gopalakrishnan, J. and Garc\'{\i}a-Castillo, L. E. and Demkowicz,
	L. F.},
	TITLE = {N\'{e}d\'{e}lec spaces in affine coordinates},
	JOURNAL = {Comput. Math. Appl.},
	FJOURNAL = {Computers \& Mathematics with Applications. An International
	Journal},
	VOLUME = {49},
	YEAR = {2005},
	NUMBER = {7-8},
	PAGES = {1285--1294},
}

@article {RapettiM2AN,
	AUTHOR = {Rapetti, F.},
	TITLE = {High order edge elements on simplicial meshes},
	JOURNAL = {ESAIM Math. Model. Numer. Anal.},
	FJOURNAL = {ESAIM Mathematical Modelling and Numerical Analysis},
	VOLUME = {41},
	YEAR = {2007},
	NUMBER = {6},
	PAGES = {1001--1020},
}

@article {DodziukCharacterization,
	AUTHOR = {Dodziuk, J.},
	TITLE = {A characterization of {W}hitney forms},
	JOURNAL = {Proc. Amer. Math. Soc. Ser. B},
	FJOURNAL = {Proceedings of the American Mathematical Society. Series B},
	VOLUME = {10},
	YEAR = {2023},
	PAGES = {455--460},
}

@article {Hesthaven,
	AUTHOR = {Hesthaven, J. S.},
	TITLE = {From electrostatics to almost optimal nodal sets for
	polynomial interpolation in a simplex},
	JOURNAL = {SIAM J. Numer. Anal.},
	FJOURNAL = {SIAM Journal on Numerical Analysis},
	VOLUME = {35},
	YEAR = {1998},
	NUMBER = {2},
	PAGES = {655--676},
}

@article {BrutmanSINUM,
    AUTHOR = {Brutman, L.},
     TITLE = {On the {L}ebesgue function for polynomial interpolation},
   JOURNAL = {SIAM J. Numer. Anal.},
  FJOURNAL = {SIAM Journal on Numerical Analysis},
    VOLUME = {15},
      YEAR = {1978},
    NUMBER = {4},
     PAGES = {694--704},
}

@article {ABRCalcolo,
	AUTHOR = {Alonso Rodr\'{\i}guez, A. and Bruni Bruno, L. and Rapetti,
	F.},
	TITLE = {Towards nonuniform distributions of unisolvent weights for
	high-order {W}hitney edge elements},
	JOURNAL = {Calcolo},
	FJOURNAL = {Calcolo. A Quarterly on Numerical Analysis and Theory of
	Computation},
	VOLUME = {59},
	YEAR = {2022},
	NUMBER = {4},
	PAGES = {Paper No. 37, 29},
}

@article {Hiemstra,
	AUTHOR = {Hiemstra, R. R. and Toshniwal, D. and Huijsmans, R. H. M. and
	Gerritsma, M. I.},
	TITLE = {High order geometric methods with exact conservation
	properties},
	JOURNAL = {J. Comput. Phys.},
	FJOURNAL = {Journal of Computational Physics},
	VOLUME = {257},
	YEAR = {2014},
	NUMBER = {part B},
	PAGES = {1444--1471},
}

@article {BruniZampa,
	AUTHOR = {Bruni Bruno, L. and Zampa, E.},
	TITLE = {Unisolvent and minimal physical degrees of freedom for the
	second family of polynomial differential forms},
	JOURNAL = {ESAIM Math. Model. Numer. Anal.},
	FJOURNAL = {ESAIM. Mathematical Modelling and Numerical Analysis},
	VOLUME = {56},
	YEAR = {2022},
	NUMBER = {6},
	PAGES = {2239--2253},
}

@article {Strang,
    AUTHOR = {Strang, G.},
     TITLE = {The fundamental theorem of linear algebra},
   JOURNAL = {Amer. Math. Monthly},
  FJOURNAL = {American Mathematical Monthly},
    VOLUME = {100},
      YEAR = {1993},
    NUMBER = {9},
     PAGES = {848--855},
}

@misc {BruniThesis,
	author = {Bruni Bruno, L.},
	title = {Weights as degrees of freedom for high order Whitney finite elements},
	publisher = {Università di Trento},
	note = {Ph.D. Thesis},
	year = {2022},
}

@book {EvansGariepy,
    AUTHOR = {Evans, L. C. and Gariepy, R. F.},
     TITLE = {Measure theory and fine properties of functions},
    SERIES = {Textbooks in Mathematics},
   EDITION = {Revised},
 PUBLISHER = {CRC Press, Boca Raton, FL},
      YEAR = {2015},
     PAGES = {xiv+299},
}

@book {Krantz,
    AUTHOR = {Krantz, S. G. and Parks, H. R.},
     TITLE = {Geometric integration theory},
    SERIES = {Cornerstones},
 PUBLISHER = {Birkh\"{a}user Boston, Inc., Boston, MA},
      YEAR = {2008},
     PAGES = {xvi+339},
}

@book {HoJo85,
    AUTHOR = {Horn, R. A. and Johnson, C. R.},
     TITLE = {Matrix analysis},
 PUBLISHER = {Cambridge University Press, Cambridge},
      YEAR = {1985},
     PAGES = {xiii+561},
}

@article {KreeftGerritsma,
    AUTHOR = {Kreeft, J. and Gerritsma, M.},
     TITLE = {Mixed mimetic spectral element method for {S}tokes flow: a
              pointwise divergence-free solution},
   JOURNAL = {J. Comput. Phys.},
  FJOURNAL = {Journal of Computational Physics},
    VOLUME = {240},
      YEAR = {2013},
     PAGES = {284--309},
}

@article {GawlikLicht,
    AUTHOR = {Gawlik, E. and Holst, M. J. and Licht, M. W.},
     TITLE = {Local finite element approximation of {S}obolev differential forms},
   JOURNAL = {ESAIM Math. Model. Numer. Anal.},
  FJOURNAL = {ESAIM. Mathematical Modelling and Numerical Analysis},
    VOLUME = {55},
      YEAR = {2021},
    NUMBER = {5},
     PAGES = {2075--2099},
}

@book {Shafarevich,
    AUTHOR = {Shafarevich, I. R. and Remizov, A. O.},
     TITLE = {Linear algebra and geometry},
 PUBLISHER = {Springer, Heidelberg},
      YEAR = {2013},
     PAGES = {xxii+526},
}

@book {Hormander,
    AUTHOR = {H\"{o}rmander, L.},
     TITLE = {The analysis of linear partial differential operators. {IV}},
    SERIES = {Classics in Mathematics},
 PUBLISHER = {Springer-Verlag, Berlin},
      YEAR = {2009},
     PAGES = {viii+352},

}

@article {Wendland,
    AUTHOR = {Wendland, H.},
     TITLE = {On the convergence of a general class of finite volume
              methods},
   JOURNAL = {SIAM J. Numer. Anal.},
  FJOURNAL = {SIAM Journal on Numerical Analysis},
    VOLUME = {43},
      YEAR = {2005},
    NUMBER = {3},
     PAGES = {987--1002},
}

@article{ArBoBo15,
  title={Finite element differential forms on curvilinear cubic meshes and their approximation properties},
  author={Arnold, D. N. and Boffi, D. and Bonizzoni, F.},
  journal={Numerische Mathematik},
  volume={129},
  number={1},
  pages={1--20},
  year={2015},
  publisher={Springer},
}

@article{HiptmairPIER,
    author = {Hiptmair, R.},
    title = {High order Whitney forms},
    journal = {Progress in Electromagnetics Research},
    year = {2001},
    volume = {32},
    issue = {2},
    pages = {271–-299},
}

@book {Kl91,
    AUTHOR = {Klimek, M.},
     TITLE = {Pluripotential theory},
    SERIES = {London Mathematical Society Monographs. New Series},
    VOLUME = {6},
      NOTE = {Oxford Science Publications},
 PUBLISHER = {The Clarendon Press, Oxford University Press, New York},
      YEAR = {1991},
     PAGES = {xiv+266},
}

@misc{Comass,
      title={On comass and stable systolic inequalities}, 
      author={J. J. Hebda and M. G. Katz},
      year={2024},
      eprint={2411.13966},
      archivePrefix={arXiv},
      primaryClass={math.DG},
      url={https://arxiv.org/abs/2411.13966}, 
      note = {To appear in Differential Geometry and Its Applications},
}

@article {AlMa16,
    AUTHOR = {Alberti, G. and Marchese, A.},
     TITLE = {On the differentiability of {L}ipschitz functions with respect
              to measures in the {E}uclidean space},
   JOURNAL = {Geom. Funct. Anal.},
  FJOURNAL = {Geometric and Functional Analysis},
    VOLUME = {26},
      YEAR = {2016},
    NUMBER = {1},
     PAGES = {1--66},
}

@article {BoNiRi24,
    AUTHOR = {Bonicatto, P. and Del Nin, G. and Rindler, F.},
     TITLE = {Existence and uniqueness for the transport of currents by
              {L}ipschitz vector fields},
   JOURNAL = {J. Funct. Anal.},
  FJOURNAL = {Journal of Functional Analysis},
    VOLUME = {286},
      YEAR = {2024},
    NUMBER = {7},
     PAGES = {Paper No. 110315, 24},
}

@book {Fa86,
    AUTHOR = {Falconer, K. J.},
     TITLE = {The geometry of fractal sets},
    SERIES = {Cambridge Tracts in Mathematics},
    VOLUME = {85},
 PUBLISHER = {Cambridge University Press, Cambridge},
      YEAR = {1986},
     PAGES = {xiv+162},
}

@book {DeRham,
    AUTHOR = {De Rham, G.},
     TITLE = {Differentiable manifolds: Forms, Currents, Harmonic Forms},
    SERIES = {Grundlehren der mathematischen Wissenschaften},
 PUBLISHER = {Springer Berlin, Heidelberg},
      YEAR = {2012},
     PAGES = {x+170},
}

@article {Reichel,
    AUTHOR = {Reichel, L.},
     TITLE = {On polynomial approximation in the uniform norm by the
              discrete least squares method},
   JOURNAL = {BIT},
  FJOURNAL = {BIT. Nordisk Tidskrift for Informationsbehandling (BIT)},
    VOLUME = {26},
      YEAR = {1986},
    NUMBER = {3},
     PAGES = {349--368},
}

@article {VianelloStability,
    AUTHOR = {Vianello, M.},
     TITLE = {Dubiner distance and stability of {L}ebesgue constants},
   JOURNAL = {J. Inequal. Spec. Funct.},
  FJOURNAL = {Journal of Inequalities and Special Functions},
    VOLUME = {10},
      YEAR = {2019},
    NUMBER = {2},
     PAGES = {49--60},
}

@article {PiazzonStability,
    AUTHOR = {Piazzon, F. and Vianello, M.},
     TITLE = {Stability inequalities for {L}ebesgue constants via
              {M}arkov-like inequalities},
   JOURNAL = {Dolomites Res. Notes Approx.},
  FJOURNAL = {Dolomites Research Notes on Approximation},
    VOLUME = {11},
      YEAR = {2018},
    NUMBER = {1},
     PAGES = {1--9},
}

@article {ReproducingKernel,
    AUTHOR = {Aronszajn, N.},
     TITLE = {Theory of reproducing kernels},
   JOURNAL = {Trans. Amer. Math. Soc.},
  FJOURNAL = {Transactions of the American Mathematical Society},
    VOLUME = {68},
      YEAR = {1950},
     PAGES = {337--404},
}

@Misc{Ma13,
  author = {Marchese, A.},
  title = {Two applications of the Theory of Currents},
  year = {2013},
  URL = {http://cvgmt.sns.it/paper/2332/},
  note = {cvgmt preprint}
}

@book {Sc99,
    AUTHOR = {Schaefer, H. H. and Wolff, M. P.},
     TITLE = {Topological vector spaces},
    SERIES = {Graduate Texts in Mathematics},
    VOLUME = {3},
   EDITION = {Second},
 PUBLISHER = {Springer-Verlag, New York},
      YEAR = {1999},
     PAGES = {xii+346},
}

@book {Bo17,
    AUTHOR = {Bogachev, V. I. and Smolyanov, O. G.},
     TITLE = {Topological vector spaces and their applications},
    SERIES = {Springer Monographs in Mathematics},
 PUBLISHER = {Springer, Cham},
      YEAR = {2017},
     PAGES = {x + 456},
}

@article {Stern,
    AUTHOR = {Stern, A.},
     TITLE = {Banach space projections and {P}etrov-{G}alerkin estimates},
   JOURNAL = {Numer. Math.},
  FJOURNAL = {Numerische Mathematik},
    VOLUME = {130},
      YEAR = {2015},
    NUMBER = {1},
     PAGES = {125--133},
}

@book {CheneyLight,
    AUTHOR = {Cheney, W. and Light, W.},
     TITLE = {A course in approximation theory},
    SERIES = {Graduate Studies in Mathematics},
    VOLUME = {101},
      NOTE = {Reprint of the 2000 original},
 PUBLISHER = {American Mathematical Society, Providence, RI},
      YEAR = {2009},
     PAGES = {xvi+359},
}

@article {Sundermann,
    AUTHOR = {S\"{u}ndermann, B.},
     TITLE = {On projection constants of polynomial spaces on the unit ball
              in several variables},
   JOURNAL = {Math. Z.},
  FJOURNAL = {Mathematische Zeitschrift},
    VOLUME = {188},
      YEAR = {1984},
    NUMBER = {1},
     PAGES = {111--117},
}

@book {Folland,
    AUTHOR = {Folland, G. B.},
     TITLE = {Real analysis},
    SERIES = {Pure and Applied Mathematics (New York)},
   EDITION = {Second},
      NOTE = {Modern techniques and their applications,
              A Wiley-Interscience Publication},
 PUBLISHER = {John Wiley \& Sons, Inc., New York},
      YEAR = {1999},
     PAGES = {xvi+386},
}

\end{document}